\documentclass[12pt,a4paper,reqno]{amsart}
  
\usepackage{amsmath}
\usepackage{amsfonts}
\usepackage{amssymb}
\usepackage{bm}
\usepackage{graphics}

\newcommand\conv{\hbox{conv}}

\newcommand\BOmega{{\bf{\Omega}}}

\newcommand\R{{\mathbb{R}}}
\newcommand\C{{\mathbb{C}}}

\newcommand\D{{\mathbf{D}}}

\newcommand\g{{\mathbf{g}}}

\renewcommand\P{{\mathbf{P}}}
\newcommand\E{{\mathbf{E}}}

\renewcommand\Im{{\operatorname{Im}}}
\renewcommand\Re{{\operatorname{Re}}}
\newcommand\eps{{\varepsilon}}
\newcommand\trace{\operatorname{trace}}

\newcommand\tr{\operatorname{trace}}
\newcommand\dist{\operatorname{dist}}
\newcommand\Span{\operatorname{Span}}

\renewcommand\a{\xi}

\newcommand\dd{\partial}

\newcommand\BI{{\mathbf I}}

\renewcommand\Pr{{\mathbf P }}

\newcommand\CE{{\mathcal E}}

\theoremstyle{plain}
  \newtheorem{theorem}[subsection]{Theorem}
  \newtheorem{conjecture}[subsection]{Conjecture}

  \newtheorem{proposition}[subsection]{Proposition}
  
  \newtheorem{lemma}[subsection]{Lemma}
  \newtheorem{corollary}[subsection]{Corollary}

\theoremstyle{remark}
  \newtheorem{remark}[subsection]{Remark}

\theoremstyle{definition}

\include{psfig}

\begin{document}

\title[Universality at the hard edge]
{Random matrices:\\ The distribution of the smallest singular values}

\author{Terence Tao}
\address{Department of Mathematics, UCLA, Los Angeles CA 90095-1555}
\email{tao@math.ucla.edu}
\thanks{T. Tao is supported by a grant from the MacArthur Foundation, and by
NSF grant DMS-0649473.}

\author{Van Vu}
\address{Department of Mathematics, Rutgers University, Piscataway NJ 08854-8019}
\email{vanvu@math.rutgers.edu}
\thanks{V.   Vu  is supported by  NSF Career Grant 0635606.}
\begin{abstract}
Let $\a$ be a real-valued random variable of mean zero and
 variance $1$. Let $M_n(\a)$ denote the $n \times n$ random matrix
 whose entries are iid copies of $\a$ and $\sigma_n(M_n(\a))$ denote the least singular value of $M_n(\a)$.
The quantity $\sigma_n(M_n(\a))^2$ is thus the least eigenvalue of the Wishart matrix $M_n M_n^{\ast}$.

  We show that (under a finite moment assumption) the probability distribution $n
  \sigma_n(M_n(\a))^2$ is {\it universal} in the sense that it
  does not depend on the distribution of $\a$. In
  particular, it converges to the same limiting distribution as
  in the special case when $\a$ is real gaussian. (The limiting distribution was computed explicitly in this case
  by Edelman.) 
  
% Our main new idea is  the so-called ``projection`` method. Roughly speaking, we
  %randomly project the matrix $M_n(\a)$ onto a $s$-dimensional subspace for some
   %medium-valued $s$
  %(or more precisely, onto the orthogonal complement of the bottom $n-s$ rows of the matrix), 
  %and exploit
  %a Berry-Ess\'een-type central limit theorem to ensure that this projection
   %is distributed like a gaussian matrix.
 % We then use  this property to compare the least singular value of of $M_{n} (\a)$ with the
  %least singular value of an $s \times s$ gaussian matrix (both properly normalized).

   We also proved a similar result  for complex-valued random variables of mean zero,
   with real and imaginary parts having variance $1/2$ and covariance
   zero. Similar results are also obtained for the joint distribution of
the bottom $k$ singular values of $M_n(\a)$ for any fixed $k$ (or
even for $k$ growing as a small power of $n$) and for rectangular
matrices.

Our approach is motivated by the general idea of   ``property testing'' from combinatorics and theoretical 
computer science. This seems to be a new approach in the study of spectra of random matrices
and combines tools from various areas of mathematics. 

\end{abstract}

\maketitle

\section{Introduction}

Let $\a$ be a  real or complex-valued random variable and
$M_n(\a)$ denote the random $n \times n$ matrix whose entries are
i.i.d. copies of  $\a$. In this paper, we always impose one of the
following two normalizations on $\a$:
\begin{itemize}
\item ($\R$-normalization) $\a$ is real-valued with $\E \a = 0$ and $\E \a^2 = 1$.
\item ($\C$-normalization) $\a$ is complex-valued with $\E \a = 0$, $\E \Re(\a)^2 = \E \Im(\a)^2 =
\frac{1}{2}$, and $\E \Re(\a) \Im(\a) = 0$.
\end{itemize}
Note in both cases $\a$ has mean zero and variance one.  A
 model example of a $\R$-normalized random variable is the \emph{real gaussian}
 $\g_\R \equiv N(0,1)$, while a model example of a $\C$-normalized random variable
 is the \emph{complex gaussian} $\g_\C$ whose real and imaginary parts are iid copies of $\frac{1}{\sqrt{2}} \g_\R$.
  Another $\R$-normalized random variable of
  interest is \emph{Bernoulli}, in which $\a$ equals $+1$ or $-1$ with an equal probability $1/2$ of each.

A basic problem in random matrix theory is to understand the
distribution of singular values in the asymptotic limit $n \to
\infty$.  Given an $m \times n$ matrix $M$, let
$$ \sigma_1(M) \geq \ldots \geq \sigma_{\min(n,m)}(M) \geq 0$$
denote the non-trivial singular values of $M$. Our paper will be
focused on the ``hard edge'' of this spectrum, and in particular
on the least singular value $\sigma_n(M)$, in the case of square matrices $m=n$, but let us begin with a brief review of some known results for the rest of the spectrum.

There is a general belief  that (under reasonable hypotheses) the limiting distributions 
concerning the spectrum of a large random matrix should be ``universal'', in the sense that they should not 
depend too strongly on the distributions of the entries of the matrix. 
In particular, one usually expects that the asymptotic statistical properties that are known for 
matrices with (independent) gaussian entries should also hold for 
matrices with more general entries (such as Bernoulli). 
A well-known conjecture (now a theorem) of this type is the Circular Law conjecture (see
\cite[Chapter 10]{BS} and \cite{tv-circular}). Another example is that of Dixon's conjectures 
(see \cite[Conjectures 1.2.1 and 1.2.2]{Mehta}).

Universality has been proved for several statistics concerning 
the random matrix model $M_{n}(\a)$.  
As is well known, the bulk distribution of the singular values of $M_n(\a)$
is governed by the \emph{Marchenko-Pastur law}. It has been shown
that  for any $t \geq 0$ and any $\R$- or $\C$-normalized $\a$,
$$ \frac{1}{n} | \{ 1 \leq i \leq n: \frac{1}{n} \sigma_i( M_n( \a ) )^2 \leq t \}
| \to \frac{1}{2\pi} \int_0^{\min(t,4)}
 \sqrt{\frac{4}{x}-1}\ dx$$
as $n \to \infty$, both in the sense of probability and in the
almost sure sense.  For more details, we refer to \cite{marchenko, pastur,
yin, BS}. (In literature, very frequently one views   $\sigma_{i } (M_{n}
 (\a))^{2} $ as the eigenvalues of the sample covariance matrix $M_{n} (\a) M_{n} (\a)^{\ast} $
 and so it is more traditional to write down 
 the limiting distributions in term of $\sigma^{2}$.) 

The next objects to consider are the extremal  singular values. 
The distribution of the largest
singular value $\sigma_1$ (and more generally, the joint distribution of the top $k$ singular values) was computed for the gaussian case 
by Johansson \cite{Joh} and Johnstone \cite{John}. This distribution is governed by the Tracy-Widom law 
(and more generally, the Airy kernel). In particular, one has 

$$\frac{\sigma_{n}^{2 } -4}{ 2^{4/3} n^{-2/3} } \rightarrow TW $$ where $TW$ denotes the Tracy-Widom 
distribution. More recently, Soshnikov \cite{soshnikov} showed that 
the same result holds for all random matrices with normalized subgaussian entries.

Now we turn to the ``hard edge'' of the
spectrum, and specifically to the least singular value
$\sigma_n(M_n(\a))$. The problem of estimating the least singular value of a random matrix 
has a long history. It first surfaced in the work 
of von Neuman and Goldstein concerning numerical inversion of large matrices 
\cite{vonG}. Later, Smale \cite{smale} made a specific conjecture about the magnitude of 
$\sigma_{n}$. Motivated by a question of  Smale, 
Edelman computed the distrubition of $\sigma_{n } (\a) $ for 
the real and complex gaussian cases $\a = \g_\R, \g_\C$ \cite{edelman}:

\begin{theorem}[Limiting distributions for gaussian models]\label{edel-thm}  For any fixed $t \geq 0$, we have
\begin{equation}\label{edu}
 \P( n \sigma_n( M_n( \g_\R ) )^2 \leq t ) = \int_0^t \frac{1+\sqrt{x}}{2\sqrt{x}} e^{-(x/2 + \sqrt{x})}\ dx + o(1)
\end{equation}
as well as the exact (!) formula
$$ \P( n \sigma_n( M_n( \g_\C ) )^2 \leq t ) = \int_0^t e^{-x}\ dx.$$
\end{theorem}

Both integrals can be computed explicitly. By change of variables, one can show 

\begin{equation} \label{eqn:explicit1}
 \int_0^t \frac{1+\sqrt{x}}{2\sqrt{x}} e^{-(x/2 + \sqrt{x})}\ dx = 1- e^{-t/2-\sqrt t}.
\end{equation}

Furthermore, it is clear that 

\begin{equation} \label{eqn:explicit2}
 \int_0^t  e^{-x}\ dx = 1- e^{-t}.
\end{equation}

In fact, 
one can compute the joint distribution of the 
bottom $k$ singular values of $M_n(\g_\R)$ or $M_n(\g_\C)$
 for any constant  $k$, as was done by Forrester\cite{forrester}.
  The formula, which involves the Bessel 
 kernel,  is more complicated and is deferred to Section \ref{section:extension}. In \cite{RR}, a different approach was proposed by 
 Rider and Ramirez, which lead to description that does not involve Bessel kernels directly. 
  Ben Arous and Peche \cite{peche} generalized Forrester's result to matrices whose entries are 
 gaussian summable.

 The error term $o(1)$ in \eqref{edu} is not explicitly stated in \cite{edelman} (it relies on an asymptotic for the Tricomi function),
 but our Theorem \ref{main-thm} below will imply that it is of the form $O(n^{-c})$ for some absolute constant $c > 0$.

The proofs of the above results  relied on  special algebraic
properties of the gaussian models $\g_\R$, $\g_\C$ (and in
particular on various exact identities  enjoyed by such models). 
For instance, Edelman's  proof used the exact joint distribution of the eigenvalues of 
$\frac{1}{n} M_{n}  M_{n} ^{\ast}$, which are available in the gaussian case:

\begin{equation} \label{jointreal} (Real \,\,\, gaussian)\,\,\,
c_{1}(n) \prod_{1 \le i < j \le n} (\lambda_i -\lambda_j) \prod_{i=1} ^{n} \lambda_{i}^{-1/2}
 \exp(- \sum_{i=1}^n \lambda_i /2)
\end{equation}

\begin{equation} \label{jointcomplex} (Complex \,\,\, gaussian)\,\,\,
c_{2} (n) \prod_{1 \le i < j \le n} |\lambda_i -\lambda_j|^2 \exp(- \sum_{i=1}^n \lambda_i/2 ).
\end{equation}

Here $c_{1}(n)$ and $c_{2}(n)$ are normalizing factors. 
It appears that this approach do not extend to the case of more general $\R$- or
$\C$-normalized models $\a$, where the above formulae are not available.

In the general setting  (and in particular in discrete cases such as  Bernoulli), it is already not trivial to show that the probability that   $\sigma_n(M_n(\a))$ is positive tends to one with $n$
(this statement is, of course, obvious in the continuous case, such as gaussian, by a dimension argument). 
This was first done by Koml\'os  \cite{Kom1,Kom2}. For more recent developments along this line
we refer to \cite{ kks, tv-sing, BVW}. These papers give better and better bounds on the rate of convergence (to one) of the probability in question, but do not give any quantitative estimate on 
$\sigma_{n}$. 

In the last few years,we have seen considerable progresses in
the problem of  estimating 
$\sigma_n$ and its tail distribution. In \cite{tv-det}, the present authors 
proved an (almost sure)  lower bound for the absolute value of the determinant  of a random Bernoulli matrix. 
As the absolute value of the determinant 
 is the product of the singular values, this result implies an (almost sure) lower bound 
of the form $\exp(-n^{1/2+o(1)})$ for the least singular value. 
A significant breakthrough was achieved by Rudelson \cite{rudelson}, who established a polynomial lower bound for  $\sigma_n(M_n(\a))$ and also tail estimates for a certain range. 
Rudelson's results were then extended by several authors \cite{rv}, \cite{rv0},
 \cite{tv-condition}, \cite{tv-condition2}, \cite{tv-circular}, \cite{tv-lsv}, \cite{tv-universal}, 
 using the machinery of Inverse Littlewood-Offord theorems, introduced in \cite{tv-condition}. 
  For instance, under the assumption of bounded fourth moment $\E |\a|^4 < \infty$, it was shown in \cite{rv0} that
$$ \P( n \sigma_n( M_n(\a) )^2 \leq t ) \leq f(t) + o(1)$$
for all fixed $t > 0$, where $f(t)$ goes to zero as $t \to 0$; similarly, in \cite{rv} it was shown that
$$ \P( n \sigma_n( M_n(\a) )^2 \geq t ) \leq g(t) + o(1)$$
for all fixed $t > 0$, where $g(t)$ goes to zero as $t \to \infty$.  Under the
stronger assumption that $\a$ is subgaussian, the lower tail estimate was improved in \cite{rv0} to
\begin{equation}\label{jump}
 \P( n \sigma_n( M_n(\a) )^2 \leq t ) \leq C t^{1/2} + c^n
\end{equation}
for some constants $C > 0$ and $0 < c < 1$ depending only on the subgaussian moments of $\a$.  At the other extreme,
with no moment assumptions on $\a$, the bound
$$ \P( n \sigma_n( M_n(\a) )^2 \leq n^{-1-\frac{5}{2}A-A^2} ) \leq n^{-A+o(1)}$$
was shown for any fixed $A > 0$ in \cite{tv-lsv}.

 A common feature of the above  mentioned results  is that they give  good
  upper and lower tail bounds on $n
\sigma_n(M_n(\a))^2$, but not the distributional law. In fact,
many papers \cite{rv, rv0, tv-condition2} are partially motivated by the
following conjecture of Spielman and Teng \cite[Conjecture 2]{spielman}.

\begin{conjecture} \label{conj:ST} Let $\a$ be the Bernoulli
random variable. Then there is a constant $0 < c<1$ such that for all $t \ge 0$

\begin{equation} \label{eqn:STconj}  \P( \sqrt n  \sigma_n( M_n( \a ) )  \leq t ) \le t  +  c^{
n}.
\end{equation}
\end{conjecture}

In this paper, we introduce a new method to study small singular values. This method is analytic in nature
and enables us to prove the {\it universality } of   the limiting distribution
of $n \sigma_n (M_n (\a))^2$.

\begin{theorem}[Universality for the least singular value]\label{main-thm}  Let $\a$ be $\R$- or $\C$-normalized,
 and suppose $\E |\a|^{C_0} < \infty$ for some sufficiently large absolute constant $C_0$.
 Then for all $t>0$, we have
\begin{equation}\label{mna}
 \P( n \sigma_n( M_n( \a ) )^2 \leq t ) = \int_0^t \frac{1+\sqrt{x}}{2\sqrt{x}} e^{-(x/2 + \sqrt{x})}\ dx + O(n^{-c})
\end{equation}
if $\a$ is $\R$-normalized, and
$$ \P( n \sigma_n( M_n( \a ) )^2 \leq t ) = \int_0^t e^{-x}\ dx + O(n^{-c})$$
if $\a$ is $\C$-normalized, where $c>0$ is an absolute constant.
The implied constants in the $O(.)$ notation depend on $\E
|\a|^{C_0}$ but are uniform in $t$.
\end{theorem}

%{\bf To Phil: Can you have a picture here; to be exact, a pair for Ber and Gau}
Figure~\ref{fig1BerGau} shows an empirical demonstration of the theorem above
for Bernoulli and for gaussian distributions.

\begin{figure}
\begin{center}
\scalebox{.5}{\includegraphics{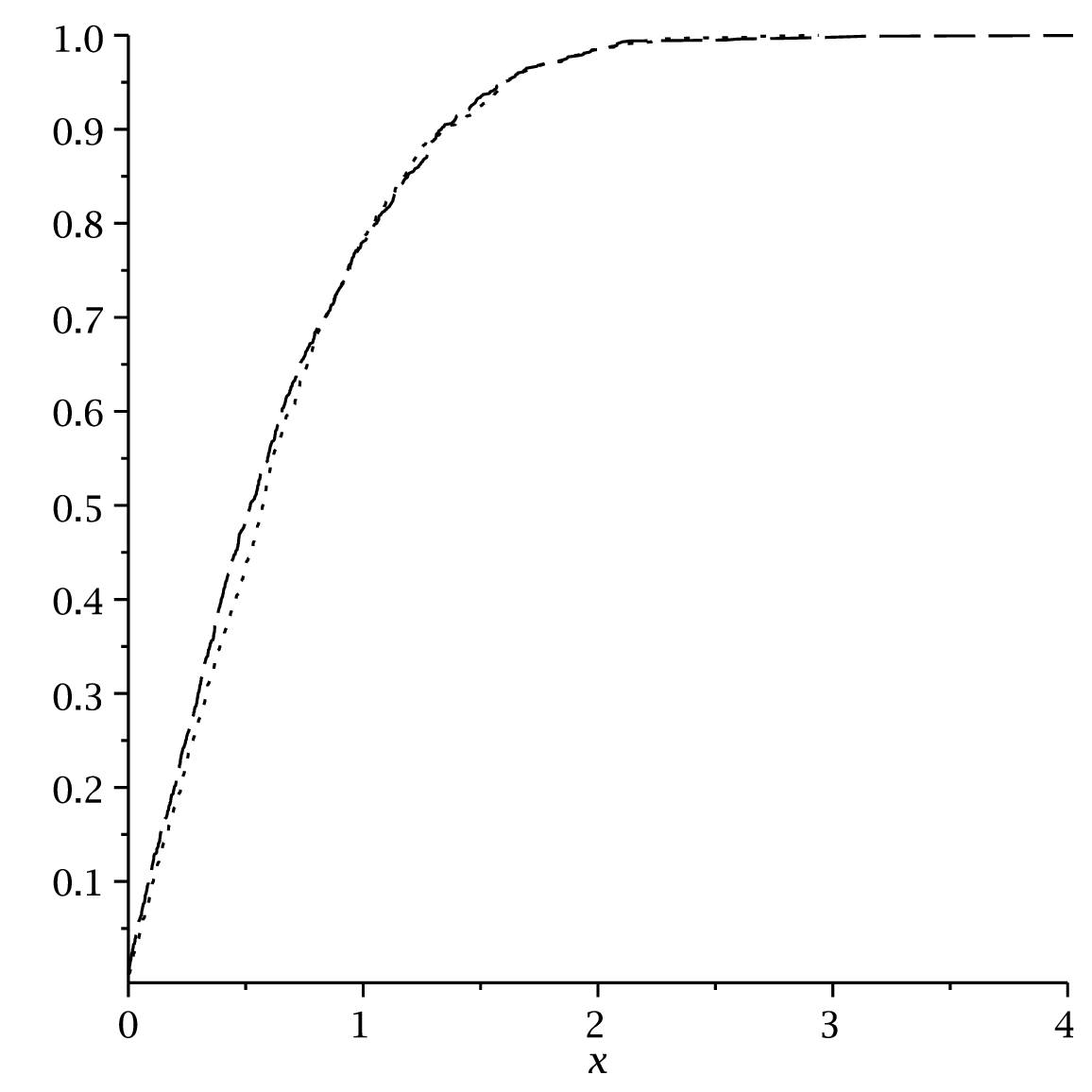}}
\end{center}
\caption{Plotted above is the curve $\Pr( \sqrt{n} \sigma_n(M_n(\xi)) \le x
)$, based on data from 1000 randomly generated matrices with $n=100$. The
dotted curve was generated with $\xi$ a random Bernoulli variable, taking the
values $+1$ and $-1$ each with probability $1/2$; and the dashed curve was
generated with $\xi$ a gaussian normal random variable.  Note that the two
curves are already close together in spite of the relatively coarse data.} 
\label{fig1BerGau}

\end{figure}

 Very roughly speaking, we will  show that one can swap $\a$ with the appropriate
gaussian distribution $\g_\R$ or $\g_\C$, at which point one can basically apply Theorem \ref{edel-thm} as a black box.
In other words, we show that the law of  $ n \sigma_n( M_n( \a ) )^2$ is {\it universal} with respect to the choice of
 $\a$ by a {\it direct comparison } to the gaussian models. The exact formulae
 $ \int_0^t \frac{1+\sqrt{x}}{2\sqrt{x}} e^{-(x/2 + \sqrt{x})}\ dx $  and
$\int_0^t e^{-x}\ dx $ do not play any important role.  This
direct comparison (or coupling) approach is in the spirit of
Lindeberg's proof \cite{lindeberg} of the central limit theorem,
 and was also recently applied 
 in our proof of the circular law \cite{tv-universal}.

Our arguments are completely effective, and give an explicit value for $C_0$; for instance, $C_0 := 10^4$
certainly suffices.  Clearly, one should be able to lower
$C_0$ significantly (we have made no attempt to optimize in $C_0$, in order to simplify the exposition),
 but we will not explore this issue here.

Theorem \ref{main-thm} can be extended in several directions, with simple modifications of the proof.  For example, we can prove a  similar universality result involving the joint distribution of the bottom
$k$ singular values of $M_n(\a)$, for bounded $k$ (and even some results when $k$ is a small power of $n$). Next, we can also consider  rectangular
 matrixes where the difference between the two dimensions is not
 too large. Finally, all results hold if we drop the condition that the entries have identical distribution. 
 (It is important that they are all normalized, independent and their $C_{0}$-moments are uniformly bounded.)  For precise statements, see Section \ref{section:extension}.

It is clear that one can use   Theorem \ref{main-thm} to
address  Conjecture \ref{conj:ST}. By Theorem \ref{main-thm} and \eqref{eqn:explicit1}, the
left hand side  of \eqref{eqn:STconj} is

\begin{align*}  \P( n  \sigma_n( M_n( \a ) ) ^{2 } \leq t ^{2})
&= \int_0^{t^{2} } \frac{1+\sqrt{x}}{2\sqrt{x}} e^{-(x/2 + \sqrt{x})}\ dx  + O(n^{-c} ) \\
&= 1- e^{-t^{2}/2 -t} + O(n^{-c}) . \end{align*}

Since $1- e^{-t^{2}/2-t } < t$ for any  $t >0$, we conclude that Conjecture \ref{conj:ST} holds for 
any $t > n^{-c_{0} }$, for some positive constant $c_{0}$. More importantly, it shows that 
the main term $t$ on the right hand side of \eqref{eqn:STconj} is only a  (first order)  approximation of the truth and could certainly be improved. For example, for sufficiently small $t$, Taylor expansion gives

$$ 1- e^{-t^{2}/2-t} \approx  t - \frac{1}{3} t^{3 } . $$

Figure~\ref{fig2Taylor} provides empirical evidence that $ \P( \sqrt{n}
\sigma_n( M_n( \a ) ) \leq t ) < t$ for the Bernoulli and gaussian cases.

\begin{figure}
\begin{center}
\scalebox{.5}{\includegraphics{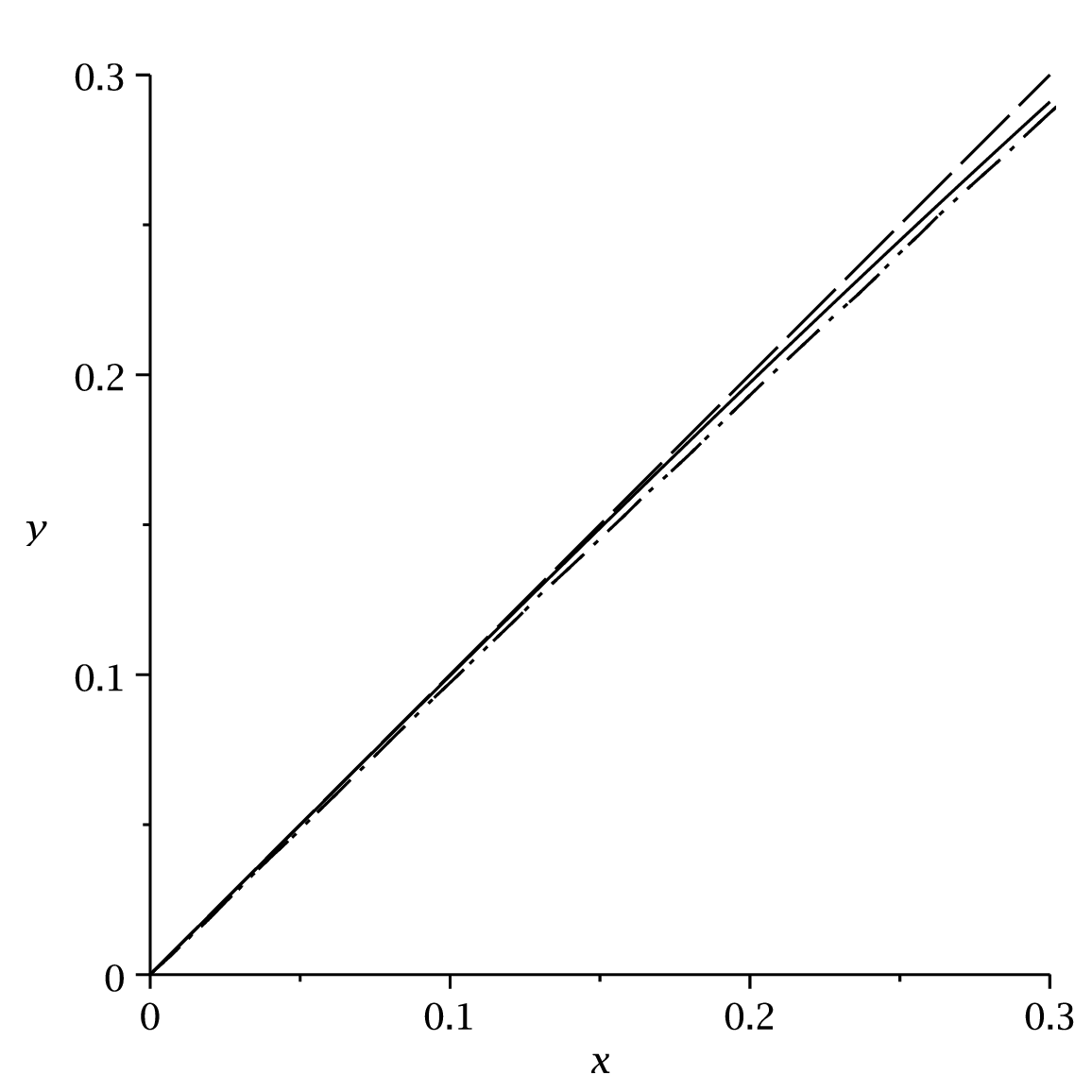}}
\end{center}
\caption{Plotted above is the curve $\Pr( \sqrt{n} \sigma_n(M_n(\xi)) \le x
)$, based on data from 150,000 randomly generated matrices with $n=100$. The
curve with long dashes is the line $y=x$, and the solid curve is a plot of
$y=x- x^3/3$.  The dotted curve was generated with $\xi$ a random Bernoulli
variable, taking the values $+1$ and $-1$ each with probability $1/2$; and the
dashed curve with spaces between the dashes was generated with $\xi$ a
gaussian normal random variable.  Note that the dotted curve and the
dashed-with-spaces curve are completely overlapping, so that it appears to be
a single cuver composed of dashes with dots inbetween.  On the right side of
the graph, the overlapping curves from the least singular values are
distinctly lower than the line $y=x$ and are very close to the curve
$y=x-x^3/3$.} 
\label{fig2Taylor}
\end{figure}
  
%{\bf To Phil: Can you have a picture here which enlarge the section near zero of the Bernoulli case to show the quadratic effect. It is good to compare with the line $x=y$.  Btw, my calculation gives $t^{2}$, do you have $t^{2} /2 $ in the Taylor expansion ?} 

  The rest of the paper is organized as follows. In the next
  section, we describe our proof strategy. In Section
  \ref{overview}, we turn this high-level strategy into a
  rigorous proof, using many technical lemmas from various areas
  of mathematics (linear algebra, theoretical computer science,
  probability and high dimensional geometry). Many of these lemmas
  may have some independent interest. For instance, Corollary \ref{cor:gaussdistance} shows that  the distance from
a random vector (in $\C^n$) to a hyperplane spanned by $n-1$ other
random vectors has (asymptotically) gaussian distribution. This is
obvious in the case when the coordinates of the vector in question are iid
gaussian, but is not so  when they are Bernoulli.  (See also \cite{rv0} for some related results in this spirit.)

{\bf Notation.} We consider $n$ as an asymptotic parameter tending
to infinity. We use $X \ll Y$, $Y \gg X$, $Y = \Omega(X)$, or $X =
O(Y)$ to denote the bound $X \leq CY$ for all sufficiently large
$n$ and for some $C$ which can depend on fixed parameters (such as
$C_0$ or $\E |\a|^{C_0}$) but is independent of $n$.

The \emph{Frobenius norm} $\|A\|_F$ of a matrix is defined as
$\|A\|_F = \tr( A A^*)^{1/2}$.  Note that this bounds the operator
norm $\|A\|_{op} := \sup \{ |Ax|: |x|=1\}$ of the same matrix.

For a random variable $X$, $\E (X)$ is the expectation of $X$. If $X$ is real, we denote by $M(X)$ its median, namely a number $x$ such that both $\P(X \ge x)$ and $\P(X \le x)$ are at least $1/2$. (If 
$x$ is not unique, choose one arbitrarily.) 

For an event $\CE$, $\BI_{\CE}$ is its indicator function, taking value 1 if $\CE$ holds and 0 otherwise. 
Clearly $\E (\BI_{\CE})= \P(\CE)$.

\section{The main idea and the proof strategy}\label{strategy}

We now discuss our main idea and the strategy behind Theorem \ref{main-thm}.
A formal version of this argument is given in Section \ref{overview}, though for various minor
technical reasons, the presentation there will be rearranged slightly from the one given here.

Let us first reveal our main idea. To start, we are going to view 
$\sigma_{n} (M_{n } (\a) )$ as the (reciprocal of the) largest singular value of $M_{n}^{-1} (\a)$.  One of the most popular methods to study the largest singular value is the moment 
method, which enables one to  control on $\sigma_{1} (M)$ (of a random matrix $M$) if  one can have good 
 estimates on $\trace (MM^{\ast})^{k/2} $ for very large $k$.
The moment method was used successfully  by Sosnhikov \cite{soshnikov}  to study $\sigma_{1} (M_{n} (\a))$.  However, it is important in the applications of this method that  the entries of $M$ are independent and their distributions well-understood. Unfortunately, 
the entries of $M_{n} ^{-1} $ are highly correlated and not much is known about their distributions.

 We have found a new approach, motivated by  the general idea of ``property testing'', a  topics popular in theoretical computer science and combinatorics. The general setting
 of a property testing problem is as follows. Given a large, complex, structure $S$, we would like to study some parameter 
 $P$ of $S$. It has been observed that quite often  one can obtain  
  good estimates  about $P$ by just looking at the small  substructure of $S$, sampled randomly. In our situation, the large structure is the matrix $S:= M_{n}^{-1} $, and the parameter in question is its largest singular value.  It has turned out that  this largest singular value can be estimated quite precisely (with high probability)  by sampling a few rows (say $s$)  from $S$ and considering the submatrix 
 $S'$ formed by these rows.  The heart of the proof then consists of two observations: (1) The singular values of $S'$ can be computed from a matrix obtained by 
 projecting the rows of $S^{-1} = M_{n} $ onto a subspace of dimension $s$ and (2) Such a projection has a central limit theorem effect.  All these together allow us to compare the least singular value of $M_{n}(\a)$ with the least singular value 
 of $M_{s} (\g)$ (properly normalized), proving the universality.

Let us now be a little more specific. 
For sake of discussion we discuss the $\R$-normalized case
\eqref{mna}. Our goal is to show
\begin{equation}\label{snn}
 \sigma_n(M_n(\a)) \approx \sigma_n(M_n(\g_\R))
\end{equation}
where we will be deliberately vague\footnote{Roughly speaking, $X \approx Y$
means that the distributions of the random variables $X$ and $Y$ are close in some appropriately
 normalized L\'evy distance, where the appropriate normalization may change from line to line.}
 as to what the symbol $\approx$ means in this non-rigorous discussion.  The $\C$-normalized case will of course be very similar.

 We have 
\begin{equation}\label{manic}
 \sigma_n(A) = \sigma_1(A^{-1})^{-1}.
\end{equation}
(From existing results, e.g. \cite{kks, tv-lsv}, it is known that
$M_n(\a)$ is invertible with very high probability.)  We now wish
to show that
$$ \sigma_1(M_n(\a)^{-1}) \approx \sigma_1(M_n(\g_\R)^{-1}).$$
Let $R_1(\a), \ldots, R_n(\a)$ denote the rows of $M_n(\a)^{-1}$, and let $s := \lfloor n^\varepsilon \rfloor$ for some small absolute constant $\varepsilon > 0$.
 To estimate the largest singular value of the $n \times n$ matrix formed by the $n$ rows
 $R_1(\a), \ldots, R_n(\a) \in \R^n$, we use random sampling.  More precisely, we create the
 $s \times n$ submatrix $B_{s,n}(\a)$ formed by selecting $s$ of these rows at random.
 Actually, since the joint distribution of $R_1(\a),\ldots,R_n(\a)$ is easily seen to be invariant
 under relabeling of the indices, we may just take $B_{s,n}(\a)$ to be the matrix with rows $R_1(\a),\ldots,R_s(\a)$.
  An application of the second moment method (see Lemma \ref{rs0} and its proof) will give us the relationship
\begin{equation}\label{sigman}
 \sigma_1(M_n(\a)^{-1}) \approx \sqrt{\frac{n}{s}} \sigma_1( B_{s,n}(\a) )
\end{equation}
provided that we have some reasonable bound on the magnitude of the rows $R_1(\a),\ldots,R_n(\a)$
(see Proposition \ref{tailbound0}); of course we expect the same statement to be true with $\a$ replaced by $\g_\R$.
  Assuming it for now, we are now (morally) reduced to establishing a relationship of the form
$$
\sigma_1( B_{s,n}(\a) ) \approx \sigma_1( B_{s,n}(\g_\R) ).$$

The next step is to use some elementary linear algebra to replace the $s \times n$ matrix $B_{s,n}(\a)$
with an $s \times s$ matrix $M_{s,n}(\a)$, defined as follows.
Let $X_1(\a), \ldots, X_n(\a) \in \R^n$ be the columns of $M_n(\a)$; observe that these iid random
variables are the dual basis of $R_1(\a),\ldots,R_n(\a)$, thus $R_i(\a) \cdot X_j(\a) = \delta_{ij}$ for $1 \leq i,j \leq n$,
where $\delta_{ij}$ is the Kronecker delta.

Let $V_{s,n}(\a)$ be the $s$-dimensional subspace of $\R^n$ defined as the orthogonal complement of the
$n-s$-dimensional space spanned by the columns $X_{s+1}(\a),\ldots,X_n(\a)$.  We select an orthonormal basis on $V_{s,n}(\a)$ arbitrarily (e.g. uniformly at random, and independently of the columns $X_1(\a),\ldots,X_s(\a)$), thus identifying $V_{s,n}(\a)$ with the standard $s$-dimensional space $\R^s$.  The orthogonal projection from $\R^n$ to $V_{s,n}(\a)$ can now be thought of as a partial isometry $\pi: \R^n \to \R^s$.  Let $M_{s,n}(\a)$ be the $s \times s$ matrix whose columns are $\pi(X_1(\a)), \ldots, \pi(X_s(\a))$.  In Lemma \ref{proj0} we will establish the simple identity
$$ \sigma_1( M_{s,n}(\a) ) = \sigma_s( M_{s,n}(\a) )^{-1},$$
thus reducing our task to that of showing that
\begin{equation}\label{sos}
 \sigma_s( M_{s,n}(\a) ) \approx \sigma_s( M_{s,n}(\g_\R) ).
\end{equation}

The relation \eqref{sos} looks very similar to our original
relation \eqref{snn} (indeed, when $s=n$, \eqref{sos} collapses
back to \eqref{snn}).  However, the critical gain here is that
\eqref{sos} becomes much easier to prove than \eqref{snn} when $s$
is only a small power of $n$, because the projection $\pi$ will
act to average out the random variable $\a$ into (approximately) a
gaussian variable (essentially thanks to the central limit
theorem). Indeed, by using a variant of the Berry-Ess\'een central
limit theorem (Proposition \ref{berry-esseen-frame0}), together
with some non-degeneracy (or ``delocalization'') properties of
$V_{s,n}(\a)$ (Proposition \ref{v-nondeg0}), we will be able to
establish a relation of the form
$$ M_{s,n}(\a) \approx M_s(g_\R)$$
and similarly
$$ M_{s,n}(\g_\R) \approx M_s(g_\R)$$
(indeed, it is not hard to see that $M_{s,n}(\g_\R)$ and $M_s(g_\R)$ in fact have an identical distribution).
The claim \eqref{sos} will then follow from the Lipschitz properties of $\sigma_s$, which is a consequence of  the Hoefmann-Weilandt theorem (see Lemma \ref{hw}).

%Perhaps surprisingly, Littlewood-Offord theorems play no role in our arguments, despite
%their prominent appearance in other places in the literature,
%see e.g. \cite{rv0}, \cite{tv-condition}, \cite{tv-condition2}, \cite{tv-circular},
%\cite{tv-lsv}.  Instead, we rely primarily on Berry-Ess\'een central limit theorems
%to take their place.  (The fact that one can substitute the latter technique for the former was also %observed in \cite{rv0}.)

\section{The rigorous proof}\label{overview}

We now implement the strategy sketched out in Section
\ref{strategy} to give a rigorous  proof of Theorem
\ref{main-thm}. This proof will rely on several key propositions
which are proven in later sections or in the appendices.

Let $F$  be either the real field $\R$ or the complex field $\C$,
and fix an $F$-normalized random variable $\a$.
  We assume $C_0$ to be a sufficiently large constant to be chosen later (e.g. $C_0 = 10^4$ certainly suffices).
   All implied constants are allowed to depend on $C_0$ and $\E |\a|^{C_0}$. In all arguments, we assume  $n$ to be large depending on these parameters.  Fix $t > 0$, and write $f(x) := \frac{1+\sqrt{x}}{2\sqrt{x}} e^{-(x/2 + \sqrt{x})}$
   if $F = \R$ or $f(x) := e^{-x}$ if $F = \C$. Our task is to show that
\begin{equation}\label{jar}
 \P( n \sigma_n( M_n( \a ) )^2 \leq t ) = \int_0^t f(x)\ dx + O(n^{-c}).
\end{equation}

We first make a simple reduction.
By hypothesis, $\E |\a|^{C_0} = O(1)$.  Hence by Markov's inequality,
we see that $\P( |\a| \leq n^{10/C_0} ) = O( n^{-10} )$.
Thus, by the union bound, we see that with probability at least
$1 - O( n^{-8} )$, all coefficients of $M_n(\a)$ are $O( n^{10/C_0} )$.
 If we remove the tail event $|\a| \geq n^{10/C_0}$ from $\a$, and readjust $\a$ slightly to
 restore the $F$-normalization conditions (using \eqref{lips} to absorb the error,
 and using the continuity of $f$), we may thus reduce to the case when
\begin{equation}\label{atoy}
|\a| \le n^{10/C_0},
\end{equation}
with probability one. This is a standard truncation and
re-normalization process, used frequently in random matrix
literature (see, for instance, \cite{BS}). We omit the (routine,
but somewhat tedious) details.

For minor technical reasons, it is also convenient to assume $\a$
to be a continuous random variable (in particular, this implies that $M_n(\a)$ is invertible with probability one).
However, we emphasise that the bounds in our arguments do not explicitly depend on the continuity properties of $\a$, and instead depend only on the $C_0$ moment of $\a$ for any fixed $n$.  Since one can express any discrete random variable with finite $C_0$ moment (and obeying \eqref{atoy}) as the limit of a sequence of continuous random variables with uniformly bounded $C_0$ moment (and also obeying \eqref{atoy}), we see that the discrete case of the theorem can be recovered from the continuous one by a standard limiting argument (keeping $n$ fixed during this process, and using \eqref{lips} as necessary).

Applying \eqref{manic}, we have
$$  \P( n \sigma_n( M_n( \a ) )^2 \leq t ) = \P( \sigma_1( M_n( \a )^{-1} )^2 \geq n/t ).$$

Let $R_1(\a), \ldots, R_n(\a)$ denote the rows of $M_n(\a)^{-1}$.
Since the columns $X_1(\a),\ldots,X_n(\a)$ of $M_n(\a)$ are exchangeable (i.e. exchanging any two columns of $M_n(\a)$
does not affect the distribution), we see that $M_n(\a)^{-1}$ is row-exchangeable.

Our first step is motivated by the observation that in certain
cases the largest singular values of a matrix can be well
approximated by sampling. This fact  is well-known in theoretical
computer science and numerical analysis. In particular, the lemma below is a
special case of more general results from \cite{Frieze, DK}.

\begin{lemma}[Random sampling]\label{rs0}  Let $1 \leq s \leq n$ be integers. $A$ be an $n \times n$ real
or complex matrix with rows $R_1,\ldots,R_n$.  Let $k_1,\ldots,k_s \in \{1,\ldots,n\}$
be selected independently and uniformly at random, and let $B$ be the $s \times n$ matrix
with rows $R_{k_1},\ldots,R_{k_s}$.  Then
$$ \E \| A^* A - \frac{n}{s} B^* B \|_F^2 \leq \frac{n}{s} \sum_{k=1}^n |R_k|^4.$$
\end{lemma}

%\begin{corollary}[Sampling of row-exchangeable matrices]\label{srm0}  Let $s, n$ be integers with $1 %\leq s \leq \sqrt{n}$,
%let $A$ be a row-exchangeable random $n \times n$ matrix with rows $R_1,\ldots,R_n$, and let $B$ %be the $s \times n$ matrix with
%rows $R_1,\ldots,R_s$.  Then
%$$ \E \sum_{i=1}^n (\sigma_i(A)^2 - \frac{n}{s} \sigma_i(B)^2)^2 \ll \frac{n^2}{s} \E (\max_{1 \leq k \leq n} |R_k|)^4.$$
%\end{corollary}

The proof of Lemma \ref{rs0} is 
presented in Appendix \ref{sampling-sec}.

In order to apply Lemma \ref{rs0}, we need to bound the right hand side. This is done in the following 
proposition.

\begin{proposition}[Tail bound on $|R_i(\a)|$]\label{tailbound0}  Let $R_{1}, \dots, R_{n}$ be the rows of $M_{n} (\a)^{-1}$. Then
$$ \P( \max_{1 \leq i \leq n} |R_i(\a)| \geq n^{100/C_0} ) \ll n^{-1/C_0}.$$
\end{proposition}

We will prove this important proposition in Section
\ref{tailbound-sec}. 

To continue, let $\CE_{1}$ denote the event

\begin{equation}\label{mori}
\max_{1 \leq i \leq n} |R_i(\a)| \leq n^{100/C_0}.
\end{equation}

By Proposition \ref{tailbound0}, we have

\begin{equation} \label{eqn:E1} 
\P(\CE_{1} ) \ge 1 - O(n^{-1/C_{0}}).
\end{equation}

Set 
\begin{equation}\label{sdef}
s := \lfloor n^{500/C_0} \rfloor.
\end{equation}

Sample a matrix $A$ from the distribution $M_{n} (\a)$. Let $B$ be the submatrix formed by $s$ random rows of $A^{-1}$.  If the rows of $A^{-1}$ satisfies 
$\CE_{1}$, then by Lemma \ref{rs0} and the definition of $s$, we have 

$$ \E (| \sigma_1( A^{-1})^2 - \frac{n}{s} \sigma_1( B)^2 |^{2})   \ll n^{-100/C_0} n^2. $$

 By Markov's inequality, 
 
$$ \P( |(\sigma_1( A^{-1} )^2 - \frac{n}{s} \sigma_1( B)^2| \geq n^{-40/C_0} n )\ll n^{-1/C_0} .$$

(The expectation and  probability in the last two estimates are 
 with respect to the random choice of the rows; the matrix $A$ is fixed and satisfies 
$\CE_{1} $.)

Let $\CE_{2}$ be the event that 

$$  |\sigma_1( A^{-1} )^2 - \frac{n}{s} \sigma_1( B)^2| \le n^{-40/C_0} n . $$

Finally, let $\CE_{3}$ be the event that 

$$\frac{n}{s}
\sigma_1( B)^2 \le n (t^{-1} - n^{-40/C_0}) . $$

We are going to view both $\CE_{2}, \CE_{3} $ as events in the product space generated by $M_{n}(\a)$ and the random choice of the rows.  A simple calculation shows  that if $\CE_{1}, \CE_{2}, \CE_{3} $ hold, then 

$$  n \sigma_n( A )^2 \ge t . $$

It follows that 

$$  \P( n \sigma_n( M_n( \a ) )^2 \ge t ) \ge \P( \frac{n}{s}
\sigma_1( B  )^2 \le  n (t^{-1} - n^{-40/C_0}) ) - O( n^{-1/C_0} ).$$

Arguing similarly, we have

$$  \P( n \sigma_n( M_n( \a ) )^2 \le t ) \ge \P( \frac{n}{s}
\sigma_1( B)^2 \geq n (t^{-1} + n^{-40/C_0}) ) - O( n^{-1/C_0} ).$$

Our next tool is a linear algebraic lemma that connects the
submatrices of $A^{-1} $ to matrices  obtained by projecting the row vectors of $A$ onto a subspace.

\begin{lemma}[Projection lemma]\label{proj0}  Let $1 \leq s \leq n$ be integers,
let $F$ be the real or complex field, let $A$ be an $n \times n$ $F$-valued invertible
matrix with columns $X_1,\ldots,X_n$, and let $R_1,\ldots,R_n$ denote the rows of $A^{-1}$.
Let $B$ be the $s \times n$ matrix with rows $R_1,\ldots,R_s$.  Let $V$ be
the $s$-dimensional subspace of $F^n$ formed as the orthogonal complement of the span
of $X_{s+1},\ldots,X_n$, which we identify with $F^s$ via an orthonormal basis,
and let $\pi: F^n \to F^s$ be the orthogonal projection to $V \equiv F^s$.
 Let $M$ be the $s \times s$ matrix with columns $\pi(X_1),\ldots,\pi(X_s)$.  Then $M$ is invertible, and we have
$$ B B^* = M^{-1} (M^{-1})^*.$$
In particular, we have
$$ \sigma_j( B ) = \sigma_{s-j+1}( M )^{-1}$$
for all $1 \leq j \leq s$.
\end{lemma}

We prove this lemma in Appendix \ref{section:algebra}.

For any fixed value of the rows $X_{s+1},\ldots,X_n$, we choose a projection $\pi: F^n \to F^s$ as above.  The exact choice of $\pi$ is not terribly important so long as it is made independently of the rows $X_1,\ldots,X_s$; for instance, one could pick $\pi$ uniformly at random with respect to the Haar measure on all such available projections, independently of $X_1,\ldots,X_s$.

Applying Lemma \ref{proj0} twice, and noticing that since the rows of $M_{n} (\a)$ have identical distribution,
we can rewrite the inequalities
in question as

$$  \P( n \sigma_n( M_n( \a ) )^2 \leq t ) \leq \P( s \sigma_s( M_{s,n}(\a) )^2 \leq t_+  ) + O( n^{-1/C_0} )$$
and
$$  \P( n \sigma_n( M_n( \a ) )^2 \leq t ) \geq \P( s \sigma_s( M_{s,n}(\a) )^2 \leq t_- ) - O( n^{-1/C_0} )$$
respectively, where
$$ t_+ := \frac{1}{ \max( t^{-1} - n^{-40/C_0}, 0 ) }$$
(with the convention that $t_+ = \infty$ if $t > n^{40/C_0}$) and
$$ t_- := \frac{1}{t^{-1} + n^{-40/C_0}}$$
and $M_{s,n}(\a)$ is the $s \times s$ matrix with rows $\pi(X_1(\a)), \ldots, \pi(X_s(\a))$, and $\pi: F^n \to F^s$ is the orthogonal projection to the $s$-dimensional space $V = V_{s,n}(\a)$ orthogonal to the columns $X_{s+1}(\a),\ldots,X_n(\a)$ of $M_n(\a)$, where we identify $V$ with $F^s$ via some orthonormal basis of $V$, chosen in some fashion independent of first $s$ columns $X_1(\a), \ldots, X_s(\a)$.

The final, and key, point in the proof is to show that the
distribution of $M_{s,n}(\a)$ is very close to that of $M_{s,n}
(\g)$. We are going to need the following high-dimensional
generalization of the classical Berry-Esseen central limit theorem, which we will prove in Appendix
\ref{section:be}.

\begin{proposition}[Berry-Ess\'een-type central limit theorem for frames]\label{berry-esseen-frame0}
 Let $1 \leq N \leq n$, let $F$ be the real or complex field,
 and let $\a$ be $F$-normalized and have finite third moment $\E |\a|^3 < \infty$.
  Let $v_1,\ldots,v_n \in F^N$ be a \emph{normalized tight frame} for $F^N$, or in other words
\begin{equation}\label{vivi}
v_1 v_1^* + \ldots + v_n v_n^* = I_N,
\end{equation}
where $I_N$ is the identity matrix on $F^N$.  Let $S \in F^N$
denote the random variable
$$ S = \a_1 v_1 + \ldots + \a_n v_n,$$
where $\a_1,\ldots,\a_n$ are iid copies of $\a$.  Similarly, let
$G := (\g_{F,1},\ldots,\g_{F,N}) \in F^N$ be formed from $N$ iid
copies of $\g_F$.  Then for any measurable set $\Omega \subset
F^N$ and any $\eps > 0$, one has

\begin{align*} 
 \P( G \in \Omega \backslash \partial_\eps \Omega ) - & O(N^{5/2} \eps^{-3} (\max_{1 \leq j \leq n} |v_j|))
\leq \P( S \in \Omega ) \\ & \leq \P( G \in \Omega \cup \partial_\eps
\Omega ) + O(N^{5/2} \eps^{-3} (\max_{1 \leq j \leq n} |v_j|)),\end{align*}
where
$$ \partial_\eps \Omega := \{ x \in F^N: \dist_\infty( x, \partial \Omega ) \leq \eps \},$$
$\partial \Omega$ is the topological boundary of $\Omega$, and and
$\dist_\infty$ is the distance using the $l^\infty$ metric on
$F^N$.  The implied constant depends on the third moment $\E
|\a|^3$ of $\a$.
\end{proposition}

In order to apply this result, we need to ensure non-degeneracy of
the subspace $V = V_{s,n}(\a)$.  This is done in the following
proposition, which we will proved in Section \ref{normal-sec}.

\begin{proposition}[$V$ is non-degenerate]\label{v-nondeg0}  Let $E_c(\a)$ denote the
event that there does not exist a unit vector $v = (v_1,\ldots,v_n)$ in $V_{s,n}(\a)$
such that $\max_{1 \leq i \leq n} |v_i| \geq n^{-c}$.  If $c$ is a sufficiently small absolute constant, then
$$\P(\overline {E_c(\a)}) \ll \exp( - n^{\Omega(1)} ).$$
\end{proposition}

 An important  point here is that $c$ does not depend on $C_0$. For instance,
  we will be able to take $c := 1/20$.

  The idea that normal vectors are
  non-degenerate was considered in \cite{tv-det} and has since
  then become an important part in several papers on the least singular value problem,
 e.g. \cite{rudelson}, \cite{rv}, \cite{rv0}, \cite{rv2}, \cite{tv-condition}, \cite{tv-condition2},
  \cite{tv-circular}, \cite{tv-lsv}. The notion of non-degeneracy
  here is, however, somewhat different from those considered
  before. In the above mentioned papers, it was typically required
  that  no small set of coordinates (say $n^{.99}$ coordinates) 
  contains most of the mass of the vector (in other words, one
  cannot compress the vector into a much shorter one).  In contrast, we require here that no individual coordinate can contain a significant (but not overwhelming) portion of the mass.

Let us take this proposition for granted for now, and condition on $$X_{s+1}(\a),\ldots,X_n(\a)$$
so that $E_c(\a)$ holds; note that $X_1(\a),\ldots,X_s(\a)$ remain iid under this conditioning.  We can then write
$$ M_{s,n}(\a) = \sum_{i=1}^s \sum_{j=1}^n \a_{ij} V_{ij}$$
where the $\a_{ij}$ are iid copies of $\a$, and $V_{ij}$ is the $s \times s$ matrix with
all rows vanishing except the $i^{th}$ row, which is equal to $\pi(e_j)$, where $e_j$ is the $j^{th}$ basis vector of $F^n$.
Since $\pi$ is a partial isometry, $\pi \pi^* = I_s$, which implies that
$$ \sum_{i=1}^s \sum_{j=1}^n V_{ij} V_{ij}^* = I_{s^2}$$
where we view identify the space of $s \times s$ matrices with the vector space $F^{s^2}$ in the standard manner.

Now let $\Omega \subset F^{s^2}$ be the set of all $s \times s$ matrices $A$ such that $s \sigma_s( A )^2 \leq t_+$, thus
$$
\P( s \sigma_s( M_{s,n}( \a ) )^2 \leq t_+ | X_{s+1}(\a),\ldots,X_n(\a) ) = \P( \sum_{i=1}^s \sum_{j=1}^n \a_{ij} V_{ij} \in \Omega | X_{s+1}(\a),\ldots,X_n(\a) ).$$
Applying Proposition \ref{berry-esseen-frame0} (with $N := s^2$) and the definition of the event $E_c$, we thus have
$$
\P( s \sigma_s( M_{s,n}( \a ) )^2 \leq t_+ |  X_{s+1}(\a),\ldots,X_n(\a) ) \leq \P( M_s(\g_F) \in \Omega \cup \partial_\eps \Omega ) + O(s^{5} \eps^{-3} n^{-c})$$
where $\eps$ is a parameter to be chosen later.  But by \eqref{lips} and the crude bound
$$\|A\|_{op} \leq \|A\|_F \leq s \sup_{1 \leq i,j \leq s} |a_{ij}|$$
for any $s \times s$ matrix $A = (a_{ij})_{1 \leq i,j \leq s}$, we see that if $A$ is an $s \times s$ matrix in $\Omega \cup \partial_\eps \Omega$, then
$$ \sigma_s( A  ) \leq \sqrt{\frac{t_+}{s}} + O( s \eps ) $$
and thus

\begin{align*} 
&\P( s \sigma_s( M_{s,n}( \a ) )^2 \leq t_+ | |  X_{s+1}(\a),\ldots,X_n(\a) ) \\ &\leq \P( M_s(\g_F)
\leq \sqrt{\frac{t_+}{s}} + O( s \eps ) ) + O( s^{5} \eps^{-3} n^{-c} ).\end{align*}
Setting $\eps := n^{-c/4}$ (say), we see from construction of $s$ that $s^{5} \eps^{-3} n^{-c} = O( n^{-1/C_0} )$, if $C_0$ is large enough depending on $c$.  Also, we can absorb the $O(s\eps)$ error into the main term $\sqrt{\frac{t_+}{s}}$ by replacing $t_+$ by a slightly larger quantity, e.g.
$$ t_{++} := \frac{1}{ \max( t^{-1} - O(n^{-30/C_0}), 0 ) }.$$
Integrating over all $X_{s+1}(\a),\ldots,X_n(\a)$ obeying $E_c$, and then using Proposition \ref{v-nondeg0} and the above argument,  we conclude that
\begin{equation}\label{nna}
\P( n \sigma_n( M_n( \a ) )^2 \leq t ) \leq \P( s \sigma_s( M_s(\g_F) )^2 \leq t_{++} ) + O( n^{-1/C_0} ).
\end{equation}
A similar argument gives
\begin{equation}\label{nnb}
\P( n \sigma_n( M_n( \a ) )^2 \leq t ) \geq \P( s \sigma_s( M_s(\g_F) )^2 \leq t_{--} ) - O( n^{-1/C_0} ).
\end{equation}
where
$$ t_{--} := \frac{1}{ \max( t^{-1} - O(n^{-30/C_0}), 0 ) }.$$

At this point we could use Theorem \ref{edel-thm} and obtain a qualitative version of Theorem \ref{main-thm}. This would result in   error terms $o(1)$ rather than $O(n^{-c})$. However, 
 one only needs a little more effort to obtain the full result.
The estimates \eqref{nna}, \eqref{nnb} hold with $n$ replaced by $2n$;
adjusting $t$ appropriately (as well as the implied constants defining $t_{--}, t_{++}$), one obtains the bounds
\begin{equation}\label{t1}
\P( n \sigma_n( M_n( \a ) )^2 \leq t ) \leq \P( 2n \sigma_{2n}( M_{2n}( \a ) )^2 \leq t_{++} ) + O( n^{-1/C_0} )
\end{equation}
and
\begin{equation}\label{t2}
 \P( n \sigma_n( M_n( \a ) )^2 \leq t ) \geq \P( 2n \sigma_{2n}( M_{2n}( \a ) )^2 \leq t_{--} ) - O( n^{-1/C_0} ).
 \end{equation}
This bound is established under the hypothesis \eqref{atoy}, but as discussed at the beginning of the section, we may
easily remove this hypothesis.  In particular, the above bounds now hold for the gaussian distribution
 $\a = \g_F$.  Meanwhile, from Theorem \ref{edel-thm} we have
\begin{equation}\label{nsoup}
\lim_{n \to \infty} \P( n \sigma_n( M_n( \g_F ) )^2 \leq t ) = \int_0^t f(x)\ dx.
\end{equation}
Iterating\footnote{One way to interpret this iteration is to view \eqref{t1}, \eqref{t2} as asserting that the random variables $n \sigma_n( M_n(\a) )^2$ form a Cauchy sequence in the L\'evy metric.} \eqref{t1}, \eqref{t2} and using \eqref{nsoup} to pass to the limit (and adjusting the implied constants in the definition of $t_{--}$, $t_{++}$ again), we conclude that
$$ \P( n \sigma_n( M_n( \g_F ) )^2 \leq t ) \leq \int_0^{t_{++}} f(x)\ dx + O( n^{-1/C_0} )$$
and
$$ \P( n \sigma_n( M_n( \g_F ) )^2 \leq t ) \geq \int_0^{t_{--}} f(x)\ dx + O( n^{-1/C_0} );$$
since $f$ is absolutely integrable and decays exponentially at infinity, we thus have
$$ \P( n \sigma_n( M_n( \g_F ) )^2 \leq t ) = \int_0^t f(x)\ dx + O( n^{-1/C_0} ).$$
Inserting this bound (with $n$ replaced by $s$) into \eqref{nna}, \eqref{nnb} and using the continuity of $f$ again,
we obtain \eqref{jar}, and Theorem \ref{main-thm}.

\begin{remark} \label{remark:main}  The above argument used the explicit limiting statistics for the hard edge of the gaussian random matrix spectrum in Theorem \ref{edel-thm}.  If one refused to use this theorem, the best one could say using the above argument is that there exists a random variable $X_F$ taking values in the positive real line and independent of $n$ and $\a$ such that
$$\P( X_F \leq t - n^{-c} ) - n^{-c}\leq \P( n \sigma_n( M_n( \a ) )^2 \leq t ) \leq \P( X_F \leq t + n^{-c}  ) + n^{-c}$$
for any $t > 0$.  (We can replace $O(n^{-c})$ by $n^{-c}$ by slightly changing the value of 
$c$.) In other words, the random variables $n \sigma_n( M_n( \a ) )^2$ converge at some polynomial rate with respect to the L\'evy metric to a universal limit $X_F$ depending only on the field $F$.  Theorem \ref{edel-thm} can then be viewed as a computation as to what this universal limit is.
\end{remark}

%It remains to prove Proposition \ref{tailbound} and Proposition \ref{v-nondeg}.  This will be done in the next two sections.

\section{Non-degeneracy of normal vectors}\label{normal-sec}

In this section we prove Proposition \ref{v-nondeg0}.  Let $c > 0$ be a sufficiently small constant.
 By the union bound and symmetry, it suffices to show that
\begin{equation}\label{v1}
 \P( |v_1| \geq n^{-c} \hbox{ for some } v = (v_1,\ldots,v_n) \in V_{s,n}, |v|=1 ) \ll \exp( - n^{\Omega(1)} ).
\end{equation}
Let $v = (v_1,\ldots,v_n)$ be such that the event in \eqref{v1} occurs.  Since $v \in V_{s,n}$, we have
$$ B v = 0$$
where $B$ is the $(n-s) \times n$ matrix with rows $X_{s+1},\ldots,X_n$.  We write $B = (Y_1, B')$,
where $Y_1 \in F^{n-s}$ is a column vector whose entries are iid copies of $\a$, and $B'$ is a $n-s \times n-1$
matrix whose entries are also iid copies of $\a$.  We similarly write $v = (v_1,v')$ where $v' := (v_2,\ldots,v_n) \in F^{n-1}$.
Then we have
\begin{equation}\label{yib}
Y_1 v_1 + B' v' = 0.
\end{equation}

Our first  tool here is the following lemma, which asserts that
with overwhelming probability, a random matrix has many small
singular values. We will prove this lemma in Appendix
\ref{section:MP}.

\begin{lemma}[Many small singular values]\label{bai-cor0}  Let $n \geq 1$, and let $\a$ be $\R$-normalized or $\C$-normalized,
 such that $|\a| \leq n^\eps$ almost surely for some sufficiently small absolute constant $\eps > 0$.
 Then there are positive constants $c, c'$ such that with probability 
 $1- \exp(-n^{\Omega (1)} )$, $M_{n} (\a)$ has at least 
 $c' n^{1-c}$ singular values in the interval $[0, n^{1/2-c}]$. 
\end{lemma}

The next tool has a geometric flavor. It asserts that the
orthogonal projection of a random vector onto a large subspace is
strongly concentrated.

\begin{lemma}[Concentration of the projection of a random vector]\label{corup0} Let $n \geq d \geq 1$ be integers,
let $K \geq 1$, let $F$ be the real or complex field, let $\a$ be $F$-normalised with $|\a| \leq K$ almost surely,
and let $V$ be a subspace of $F^n$ of dimension $d$.  Let $X := (\a_1,\ldots,\a_n)$, where $\a_1,\ldots,\a_n$ are iid
copies of $\a$.  Suppose that $d > CK^2$ for some sufficiently large absolute constant $C$.  Then
$$ \P( |\pi_V(X) - \sqrt{d}| \geq \sqrt{d}/2 ) \ll \exp( - \Omega( d / K^2 ) ).$$
\end{lemma}

This lemma is an extension of  the special case considered in
\cite{tv-det}, in which  $\a$ is Bernoulli. We are going to prove
this lemma in Appendix \ref{section:concentration}.

Applying Lemma \ref{bai-cor0} to an $(n-s) \times ( n-s)$ block of $B$
and using the interlacing law (Lemma \ref{cauchy}), we conclude
(given that $C_0$ is sufficiently large) that with probability $1
- \exp( - n^{\Omega(1)} )$, $B'$ has at least $n^{1-c_0}$ singular
values of size $O( n^{\frac{1}{2} - c_0} )$, for some absolute
constant $c_0
> 0$. This implies the existence of a subspace $V \subset F^{n-s}$
of dimension $d := \lfloor n^{1-c_0} \rfloor$ such that $\| \pi_V
B' \|_{op} \ll n^{\frac{1}{2} - c_0}$, where $\pi_V$ is the
orthogonal projection to $V$.
 Let us condition on the event that this is the case (note that this event does not depend on $v$). Then we have
$$ | \pi_V B' v' | \ll n^{\frac{1}{2} - c_0}.$$
Combining this with \eqref{yib} and the hypothesis $|v_1| \geq n^{-c}$, we conclude that
$$ |\pi_V(Y_1)| \ll n^{\frac{1}{2} - c_0 + c}.$$

On the other hand, from Lemma \ref{corup0} we see that with
probability $1 - \exp( -n^{\Omega(1)})$, we have
$$ |\pi_V(Y_1)| \sim \sqrt{d} \sim n^{\frac{1}{2} - c_0/2}.$$
If we choose $c$ sufficiently small depending on $c_0$, we obtain a contradiction with probability
$1 - \exp( -n^{\Omega(1)})$.  The claim \eqref{v1} follows.

\section{A lower bound for  the distance between 
a random vector and a random hyperplane}\label{tailbound-sec}

In this section we prove Proposition \ref{tailbound0}.  This type of result is
 somewhat related to the lower tail bounds on the lowest singular value
 in the literature (e.g. \cite{litvak}, \cite{rudelson}, \cite{rv0}, \cite{rv2},
 \cite{tv-condition}, \cite{tv-condition2}, \cite{tv-lsv}).
 There are two basic means to obtain such bounds, namely via Littlewood-Offord theorems and via
 Berry-Ess\'een-type central limit theorems.  We will rely solely on the second route (cf. \cite{litvak}, \cite[Section 2.3]{rv0}).

 Let us begin with a simple lemma.

 \begin{lemma}[Distance from vector to hyperplane controls inverse]\label{dvh}
 Let $n \geq 1$, let $F$ be the real or complex field, let $A$ be an $n \times n$ $F$-valued invertible
  matrix with columns $X_1,\ldots,X_n$, and let $R_1,\ldots,R_n$ denote the rows of $A^{-1}$.  Then for any $1 \leq i \leq n$, we have
$$ |R_i| = \dist( X_i, V_i )^{-1}$$
where $V_i$ is the hyperplane spanned by
$X_1,\ldots,X_{i-1},X_{i+1},\ldots,X_n$.
\end{lemma}

One can directly verify this lemma using the identity $AA^{-1}
=I$. It is also the special case $s=1$ of Lemma \ref{proj0}.

By Lemma \ref{dvh}, the  desired proposition is equivalent to the
lower tail estimate
\begin{equation}\label{pmin}
 \P( \min_{1 \leq i \leq n} d_i \leq n^{-100/C_0} ) \ll n^{-1/C_0}
\end{equation}
where $d_i := \dist(X_i,V_i)$, $X_i := X_i(\a)$, $V_i :=
\Span(X_1,\ldots,X_{i-1},X_{i+1},\ldots,X_n)$.

To start, we first study the distribution of $d_1$.  Let us
temporarily condition on (i.e. freeze) the vectors  $X_2,\ldots,X_n$, leaving
$X_1$ random, and let $(v_1,\ldots,v_n)$ be a unit normal vector  to the
hyperplane spanned by $X_2,\ldots,X_n$.  Applying Proposition
\ref{v-nondeg0}, we see that with probability $1 - \exp(
-n^{\Omega(1)})$, we have $\max_{1 \leq i \leq n} |v_i| \leq
n^{-c}$ for some absolute constant $c > 0$.  Suppose that this
event occurs.  If we write $X_1 = (\a_1,\ldots,\a_n)$, where
$\a_1,\ldots,\a_n$ are iid copies of $\a$, then we have
$$ d_1 = | \a_1 v_1 + \ldots + \a_n v_n |.$$
Applying\footnote{When $F=\R$, one could also use the classical
Berry-Ess\'een theorem, Proposition \ref{berry}.} Proposition
\ref{berry-esseen-frame0}, we thus see that for any $t > 0$ and
$\eps > 0$, we have
$$ \P( d_1 \leq t ) \leq \P( |\g_F| \leq t + O(\eps) ) + O( \eps^{-3} n^{-c} );$$
setting $\eps := n^{-c/4}$ we conclude
$$ \P( d_1 \leq t ) \leq \P( |\g_F| \leq t ) + O( n^{-c/4} ).$$
We can bound $\P( |\g_F| \leq t )$ by $O(t)$ (when $F=\C$, we have
the stronger bound of $O(t^2)$, though we will not exploit this).
By symmetry, we thus obtain the lower tail estimate

\begin{equation}\label{lower-tail}
\P( d_i \leq t ) \ll t + n^{-c/4}.
\end{equation}

Notice that a slight modification of the above argument also yields the following corollary:

\begin{corollary}[Random distance is gaussian]\label{cor:gaussdistance} 
Let $X_1, \dots, X_n$ be random vectors whose
entries are iid copies of $\a$. Then  the distribution of the
distance $d_1$ from $X_1$ to $\Span (X_2, \dots, X_n)$ is approximately
gaussian, in the sense that
$$ \P( d_1 \leq t ) = \P( |\g_F| \leq t ) + O( n^{-c/4} ).$$
\end{corollary}

A naive application of the union bound to \eqref{lower-tail} is
clearly insufficient to establish \eqref{pmin}, since we have $d$
distances to deal with. However, we can at least use that bound to
show that
$$
\P( \min_{1 \leq i \leq L} d_i \leq t ) \ll L ( t + n^{-c/4} )
$$
for any $1 \leq L \leq n$.

  The key fact that enables us to overcome the  ineffectiveness
  of the  union bound is that the distances $d_{i}$ are correlated.
  They tend to be large or small at the same time. Quantitatively, we have

\begin{lemma}[Correlation between distances]\label{Relation} Let $n \geq 1$, let $F$
 be the real or complex field, let $A$ be an $n \times n$ $F$-valued invertible matrix with columns
 $X_1,\ldots,X_n$, and let $d_i := \dist(X_i,V_i)$ denote the distance from $X_i$ to the hyperplane
 $V_i$ spanned by $X_1,\ldots,X_{i-1},X_{i+1},\ldots,X_n$.  Let $1 \leq L < j \leq n$, let $V_{L,j}$
 denote the orthogonal complement of the span of $X_{L+1},\ldots,X_{j-1},X_{j+1},\ldots,X_n$, and let
 $\pi_{L,j}: F^n \to V_{L,j}$ denote the orthogonal projection onto $V_{L,j}$.  Then
$$ d_j \geq \frac{ | \pi_{L,j}(X_j) | }{ 1 + \sum_{i=1}^L \frac{|\pi_{L,j}(X_i)|}{d_i} }.$$
\end{lemma}

We are going to prove this lemma in Appendix
\ref{section:correlation}.

Fix $L = \lfloor n^{20/C_0} \rfloor$ and $t := n^{-21/C_0}$; then (if $C_0$ is sufficiently large compared with $c$) we have
\begin{equation}\label{ltail}
\P( \min_{1 \leq i \leq L} d_i \geq n^{-21/C_0} ) \geq 1 - O(n^{-1/C_0}).
\end{equation}
To control the other distances $d_j$ for $L < j \leq n$, we use Lemma \ref{Relation}, which gives the lower bound
\begin{equation}\label{tl}
d_j \geq \frac{ | \pi_{L,j}(X_j) | }{ 1 + \sum_{i=1}^L |\pi_{L,j}(X_i)|/d_i }
\end{equation}
where $\pi_{L,j}$ is the orthogonal projection to the $L+1$-dimensional space $V_{L,j}$ orthogonal to $X_{L+1},\ldots,X_{j-1},X_{j+1},\ldots,X_n$.

From Lemma \ref{corup0} and \eqref{atoy}, we have
$$ \P( |\pi_{L,j}(X_i) - \sqrt{L+1}| \geq \sqrt{L+1}/2 ) \ll \exp( - n^{1/C_0} )$$
(say) for all $1 \leq i \leq L < j \leq n$.  By the union bound, we thus have
\begin{equation}\label{utail}
\P( |\pi_{L,j}(X_i) - \sqrt{L+1}| \leq \sqrt{L+1}/2 \hbox{ for all } 1 \leq i \leq L < j \leq n ) \geq 1 - O(n^{-1/C_0})
\end{equation}
(say).    If the events in \eqref{ltail}, \eqref{utail} both occur, then we conclude from \eqref{tl} that
$$ d_j \gg n^{-41/C_0}$$
for all $L < j \leq n$, and \eqref{pmin} follows.  The proof of Proposition \ref{tailbound0} is complete.

\section{Extensions of Theorem \ref{main-thm}} \label{section:extension}

In this section, we describe several extension of Theorem \ref{main-thm}, mentioned briefly at the end of the introduction.  As an application, we show that  our results  determine the 
distribution of the condition number of $M_{n} (\a)$, extending results of Edelman from \cite{edelman}.

\subsection{Joint distribution of least singular
values}\label{joint-sec}

The proof of Theorem \ref{main-thm} can be adapted to control the joint distribution of the bottom $k$ singular values of $M_n(\a)$, for $k$ a small power of $n$.  More precisely, define $\Lambda_{k,n}(\a) \in (\R^+)^k$ to be the random variable
$$ \Lambda_{k,n}(\a) := ( n \sigma_n( M_n( \a ) )^2, \ldots, n \sigma_{n-k+1}( M_n(\a) )^2 ).$$
Then a routine modification of the proof of Theorem \ref{main-thm} gives

\begin{theorem}[Universality for $\Lambda_{k,n}(\a)$]\label{main-thm2}  Let $F$ be the real or complex field, let $\a$ be $F$-normalized, and suppose $\E |\a|^{C_0} < \infty$ for some sufficiently large absolute constant $C_0$.  If $c>0$ is a sufficiently small absolute constant, then we have
$$
\P( \Lambda_{k,n}(\g_F) \in \Omega \backslash \partial_{n^{-c}} \Omega ) - n^{-c}
\leq
 \P( \Lambda_{k,n}(\a) \in \Omega ) \leq  \P( \Lambda_{k,n}(\g_F) \in \Omega \cup \partial_{n^{-c}} \Omega ) +n^{-c}$$
for all $1 \leq k \leq n^c$ and all measurable $\Omega \in F^k$, where $\partial_{n^{-c}}\Omega$ is the set of all points in $F^k$ which lie within
$n^{-c}$ (in $l^\infty$ norm) of the topological boundary of $\Omega$.
\end{theorem}

Indeed, one simply repeats the arguments for Theorem \ref{main-thm} (which is basically the $k=1$ case) but acquires losses of $k^{O(1)}$ throughout the argument, which will be acceptable if we force $k$ to be a sufficiently small power of $n$, and if one shrinks $c$ as necessary; we omit the details.

Now we focus on the case when $k=O(1)$ is independent of $n$, then Theorem \ref{main-thm2} asserts that $\Lambda_{k,n}(\a)$ converges in the ($k$-dimensional) L\'evy metric to a limiting distribution on $F^k$.  The precise value of this limiting distribution is complicated to state, but can in principle be computed from the following asymptotics of Forrester \cite{forrester} in the complex case $a=\g_\C$:

\begin{theorem}[Bessel kernel asymptotics at the hard edge, complex case]\label{comp}\cite{forrester} Let $m \geq 1$ be an integer, and let $f \in L^\infty(\R^m)$ be a symmetric function. Then the expression
$$ \E \sum_{1 \leq i_1 < \ldots < i_m \leq n} f( 4n \sigma_{i_1}(M_n(\g_\C))^2, \ldots, 4n \sigma_{i_m}(M_n(\g_F))^2 )$$
converges as $n \to \infty$ to the limit
$$ \int_{\R^n} f(t_1,\ldots,t_m) \det( K_0(t_i,t_j) )_{1 \leq i,j \leq m}\ dt_1 \ldots dt_m$$
where $K_0$ is the \emph{Bessel kernel}
$$ K_0(x,y) := \frac{J_0(x^{1/2}) y^{1/2} J_{1}(y^{1/2}) - J_1(y^{1/2}) x^{1/2} J_{0}(x^{1/2})}{2(x-y)},$$
where $J_\nu$ is the Bessel function
$$ J_\nu(x) = \frac{1}{2\pi} \int_{-\pi}^\pi e^{-i(\nu t - x \sin t)}\ dt.$$
\end{theorem}

In the real case $a = \g_\R$, there is a similar formula, but $K_0$ is a now a $2 \times 2$ matrix-valued kernel, also defined in terms of Bessel functions, which is somewhat more complicated to state; see \cite{nagao} for details.

As mentioned above, these $k$-point correlation asymptotics are, in principle, sufficient to reconstruct the asymptotic distribution of $\Lambda_{k,n}(\a)$.  For instance, this is done for $k=1$ in \cite{forrester2} (in particular, recovering the results in Theorem \ref{edel-thm}), and for $k=2$ and $F=\C$ in \cite{forrester3}. However, the computations rapidly increase in complexity as $k$ increases, and a closed-form expression for the limit in the general $k$ case does not appear to currently be in the literature.  On the other hand, one can apply Theorem \ref{main-thm2} directly to results such as Theorem \ref{comp} to conclude that the same asymptotics are in fact valid for all $\C$-normalized $\a$ (and similarly that the asymptotics in \cite{nagao} are valid for all $\R$-normalized $\a$); we omit the details.

\begin{figure}
\centerline{{\bf Bernoulli}\qquad\qquad\qquad\qquad\qquad\qquad {\bf Gaussian}}
\begin{center}
\scalebox{.3}{\includegraphics{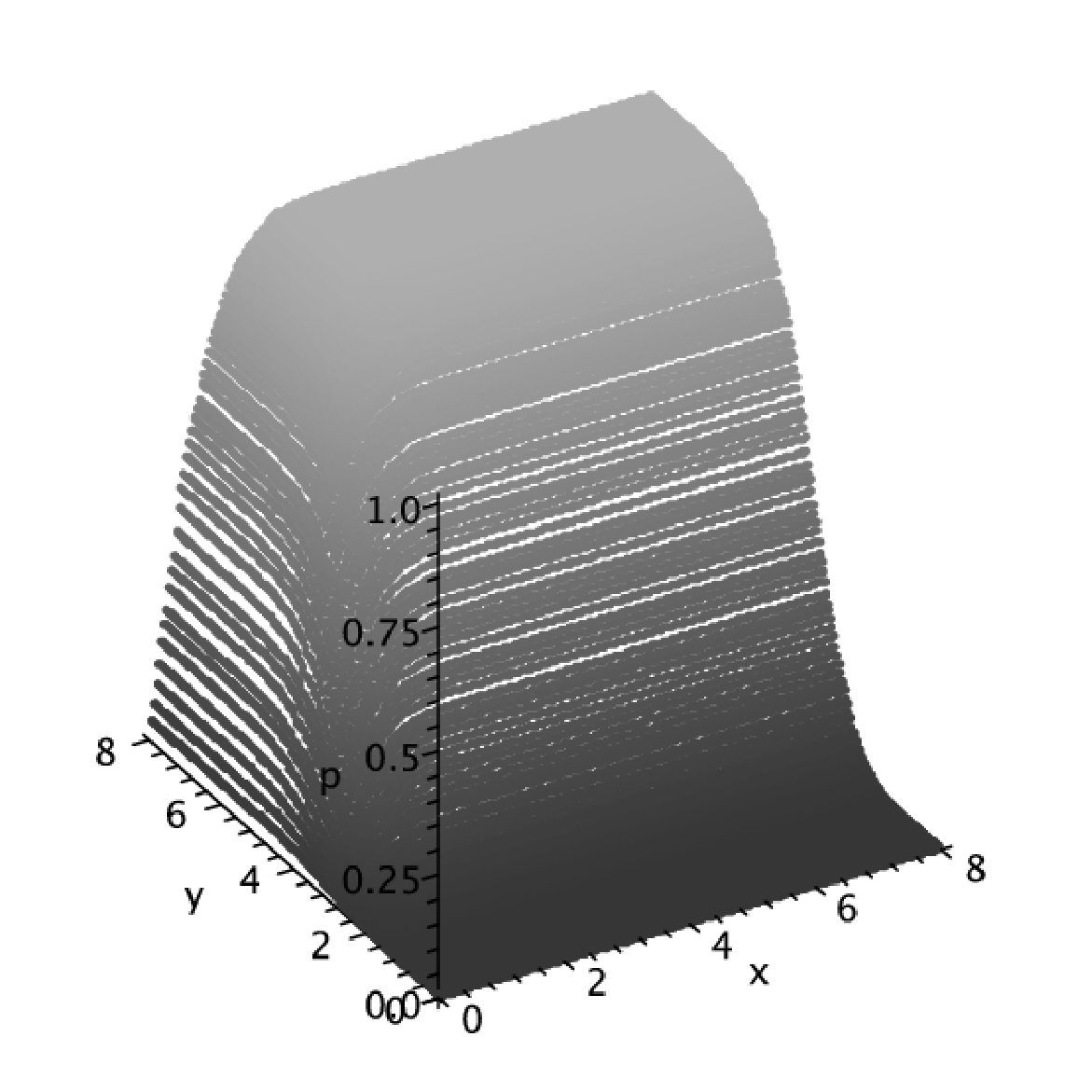}}%
\qquad\scalebox{.3}{\includegraphics{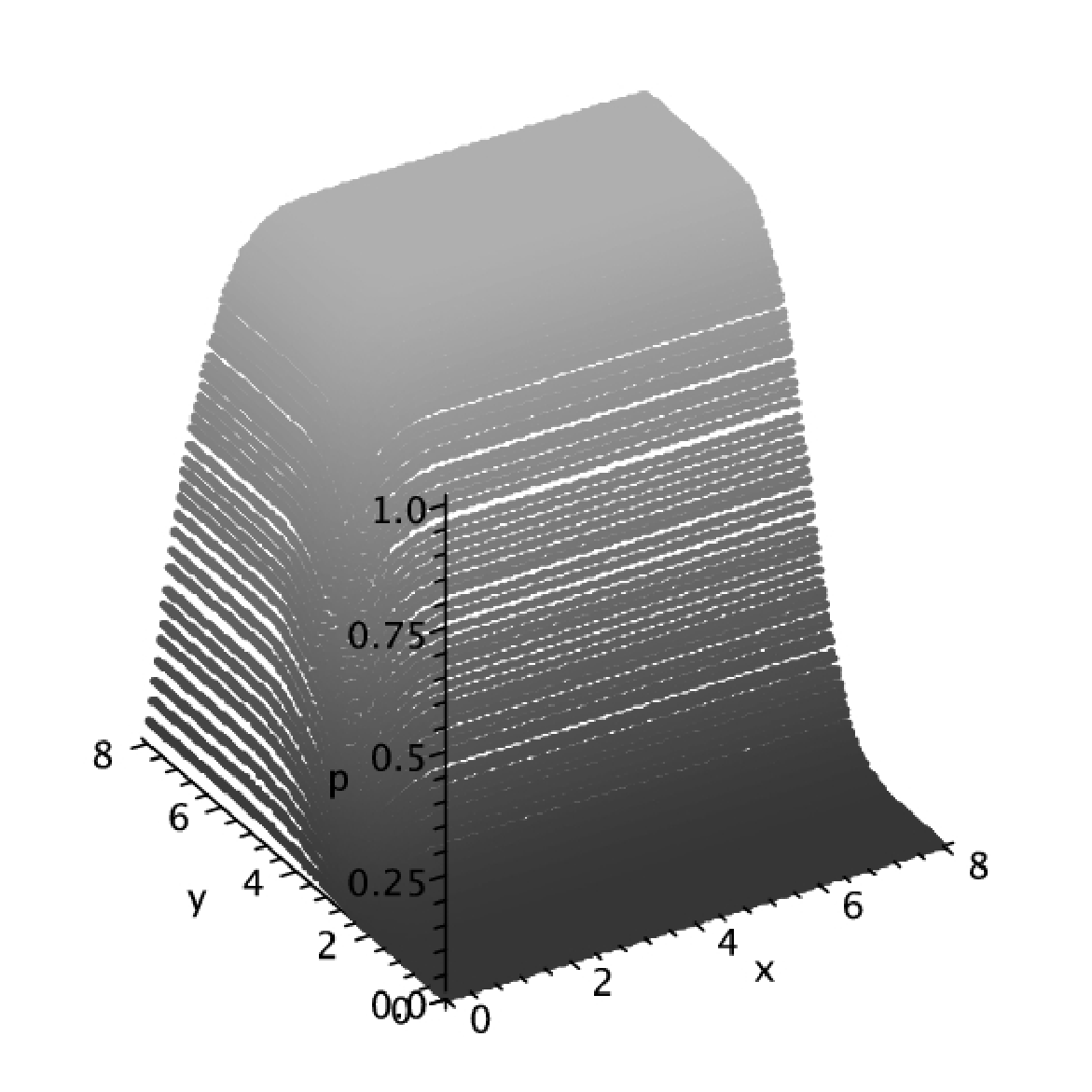}}
\end{center}
\caption{Plotted above is the surface $\Pr( \sqrt{n} \sigma_n(M_n(\xi)) \le x \mbox{ and } \sqrt{n} \sigma_{n-1}(M_n(\xi)) \le y)$, based on data from 1000 randomly generated matrices with $n=100$.   Lighter shading indicates higher probability (on the vertical axis).  The surface on the left was generated with $\xi$ a random Bernoulli
variable, taking the values $+1$ and $-1$ each with probability $1/2$; and surface on the right was generated with $\xi$ a
gaussian normal random variable. } 
\label{fig3_3d}
\end{figure}
  
Figure~\ref{fig3_3d} shows the joint distribution of the smallest two singular values for Bernoulli and gaussian distributions.

\begin{figure}
\centerline{{\bf Bernoulli}\qquad\qquad\qquad\qquad\qquad\qquad {\bf Gaussian}}
\begin{center}
\scalebox{.3}{\includegraphics{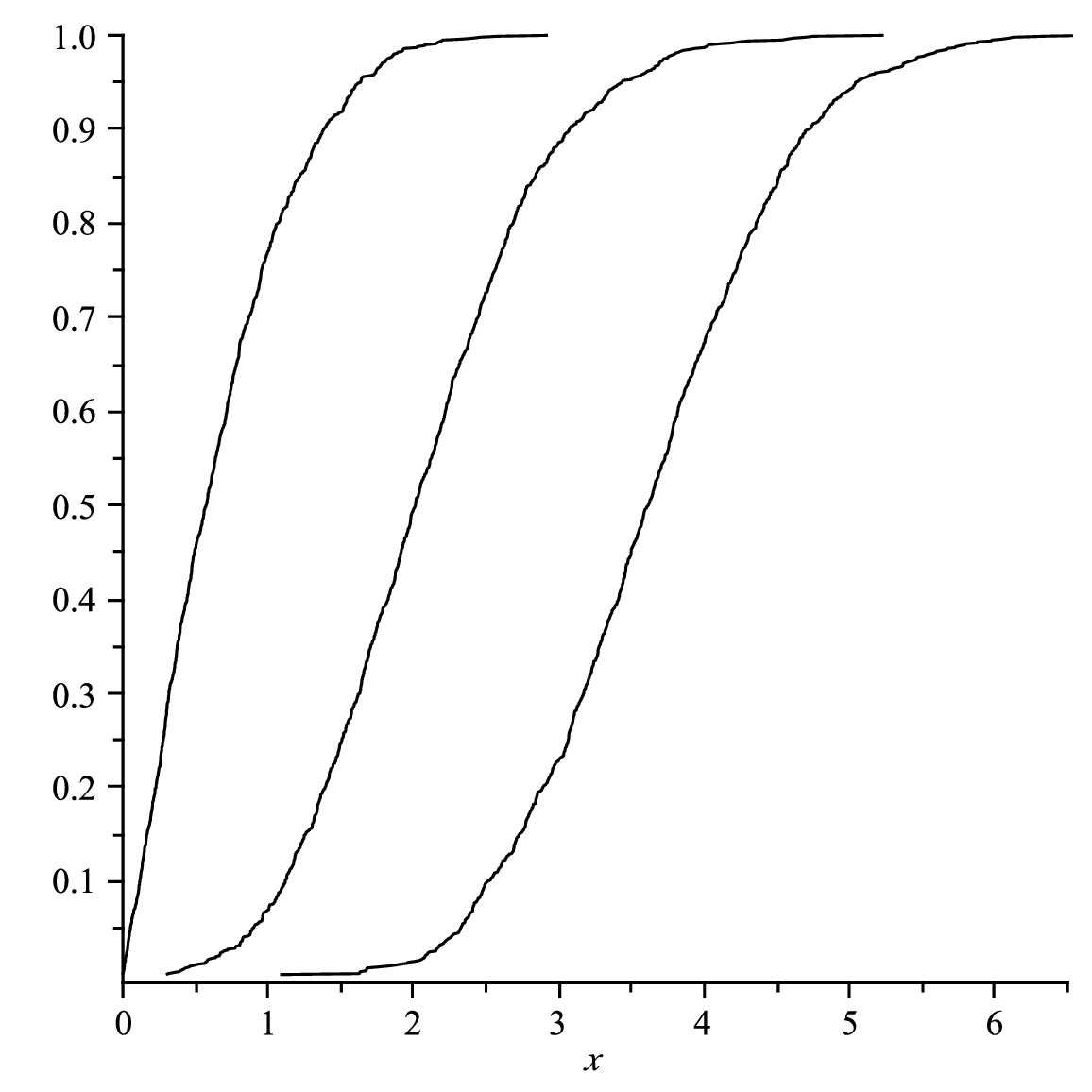}}%
\qquad\scalebox{.3}{\includegraphics{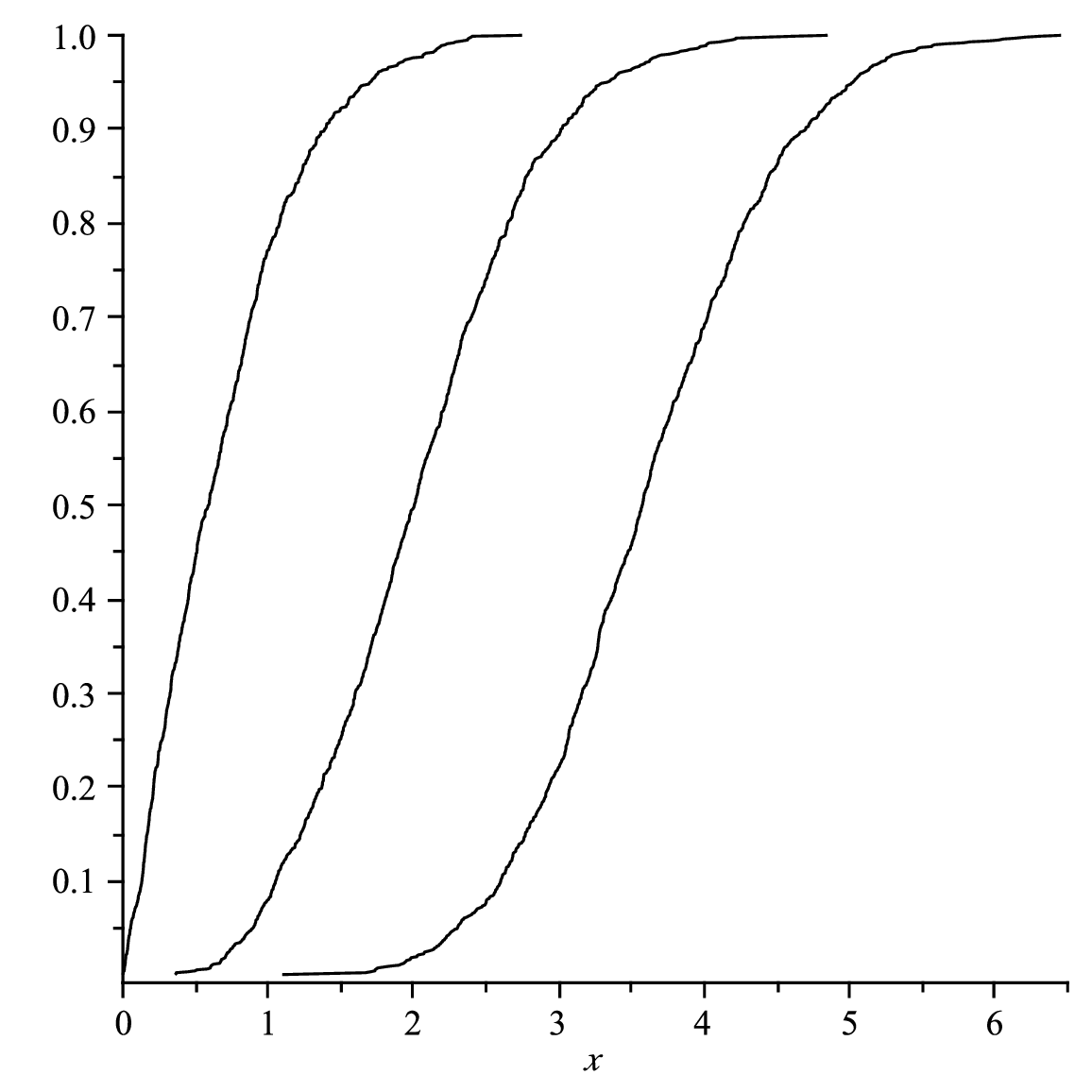}}
\end{center}\caption{Plotted above are the curves $\Pr( \sqrt{n} \sigma_{n-k} (M_n(\xi)) \le x)$, for $k=0,1,2$ based on data from 1000 randomly generated matrices with $n=100$.  The curves on the left were generated with $\xi$ a random Bernoulli
variable, taking the values $+1$ and $-1$ each with probability $1/2$; and curves on the right were generated with $\xi$ a
gaussian normal random variable.  In both cases, the curves from left to right correspond to the cases $k=0,1,2$, respectively. } 
\label{fig4_2d}
\end{figure}

Figure~\ref{fig4_2d} plots the functions $\Pr( \sqrt{n}\sigma_{n-k} (M_n(\xi)) \le x)$, where $k=0,1,2$, for Bernoulli and gaussian distributions.
%here
%{\bf To Phil: Can you have  a  picture here, comparing the 2nd or 3rd smallest singular values. We can have the smallest,  2nd, 3rd smallest of Bernoulli on one side and Gaussian on the other. I dont know how hard it is to have 3D picture, for instance for the joint distribution of say 2 singular values. } 

\subsection{Rectangular matrices}

We can use our method to consider the least singular value of a
rectangular matrix of the dimension $(n-l) \times n$, where $l$ is
a constant.   In fact, we can reduce this problem to the square
case by adding a (dummy) set of $l$ rows vectors which form an
orthonormal complement of the $n-l$ row vectors of the matrix. We
consider these dummy vectors as the first $l$ vectors of the (now
square) matrix and repeat the proof. The additional vectors remain
the same after the projection and by removing them we obtain an
$(s-l) \times s $ matrix which is asymptotically gaussian. Thus,
we will be able to compare the least singular value of the
original matrix with this one.

The limiting  distribution of the least singular value of  $(n-l) \times n $  random matrices with gaussian entries  (for $l$ constant) can be computed, at least in principle. However, we cannot find a reference for this. Thus, we are going to phrase the result here in the spirit of 
Remark  \ref{remark:main}.  To this end,  $M_{n-l,n} (\a)$ denotes the 
$(n-l) \times n $ random matrix whose entries are iid copies of $\a$. 

\begin{theorem}[Universality for the least singular value of rectangular matrices]
 Let $\a$ be $\R$- or $\C$-normalized,
 and suppose $\E |\a|^{C_0} < \infty$ for some sufficiently large absolute constant $C_0$.
 Let $l$ be a constant. 
 Let $X:= n \sigma_{n-l} M_{n-l,n} (\a)^{2}$, $X_{\g_{\R}} :=  n \sigma_{n-l} M_{n-l,n} (\g_{R})^{2}$
 and  $X_{\g_{\C}} :=  n \sigma_{n-l} M_{n-l,n} (\g_{C})^{2}$. Then there is a constant $c>0$ such that  for any $t \ge 0$
 
 $$\P( X \leq t - n^{-c}  ) - n^{-c} \leq \P( X_{\g_{R}} \leq t ) \leq \P( X  \leq t + n^{-c}  ) + n^{-c}$$
 
 if $\a$ is $\R$-normalized and 
 
 $$\P( X \leq t - n^{-c} ) - n^{-c}\leq \P( X_{\g_{\C}} \leq t ) \leq \P( X \leq t + n^{-c} ) + n^{-c}$$
 
 if $\a$ is $\C$-normalized. 
 \end{theorem}

\begin{figure}
\centerline{{\bf Bernoulli}\qquad\qquad\qquad\qquad\qquad\qquad {\bf Gaussian}}
\begin{center}
\scalebox{.3}{\includegraphics{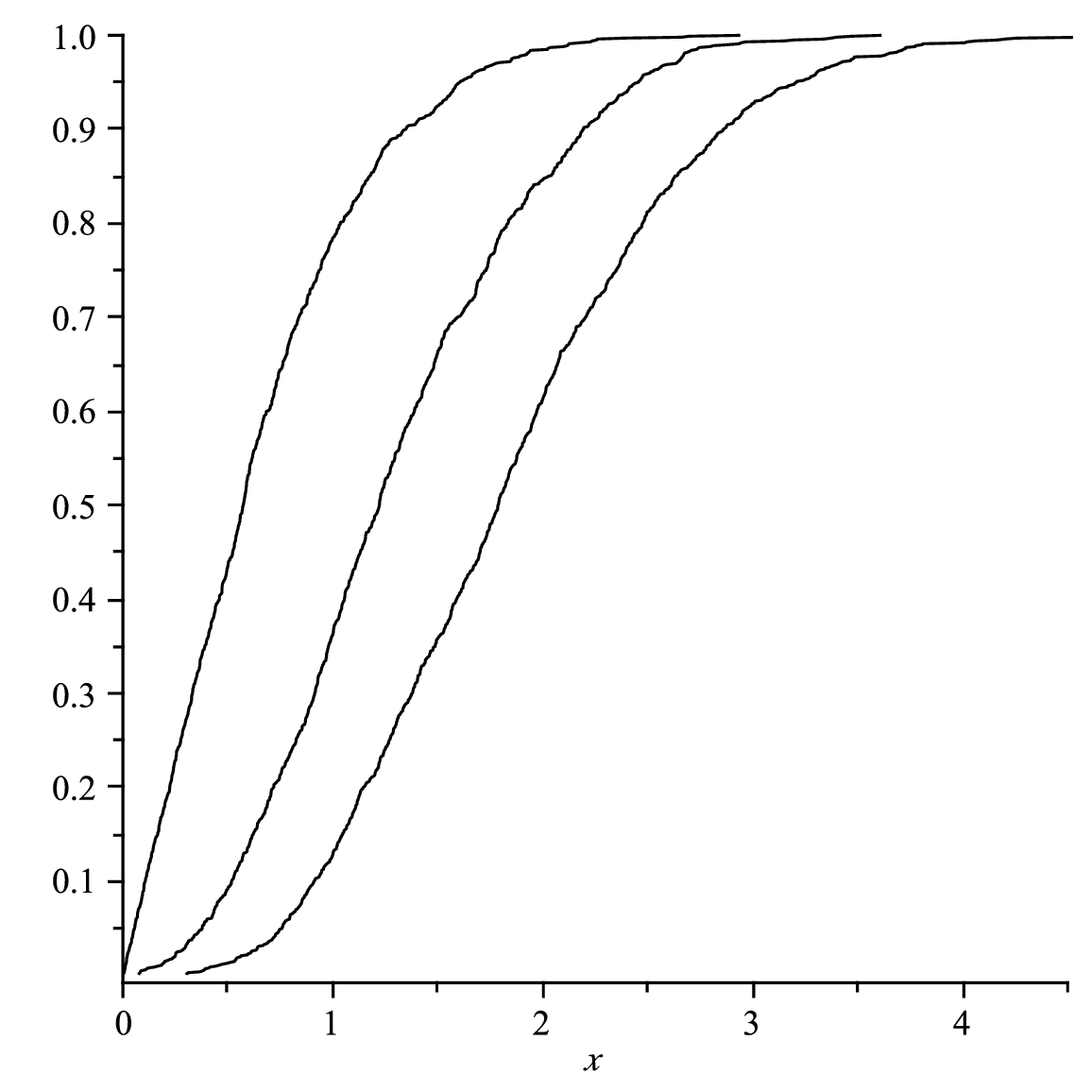}}%
\qquad\scalebox{.3}{\includegraphics{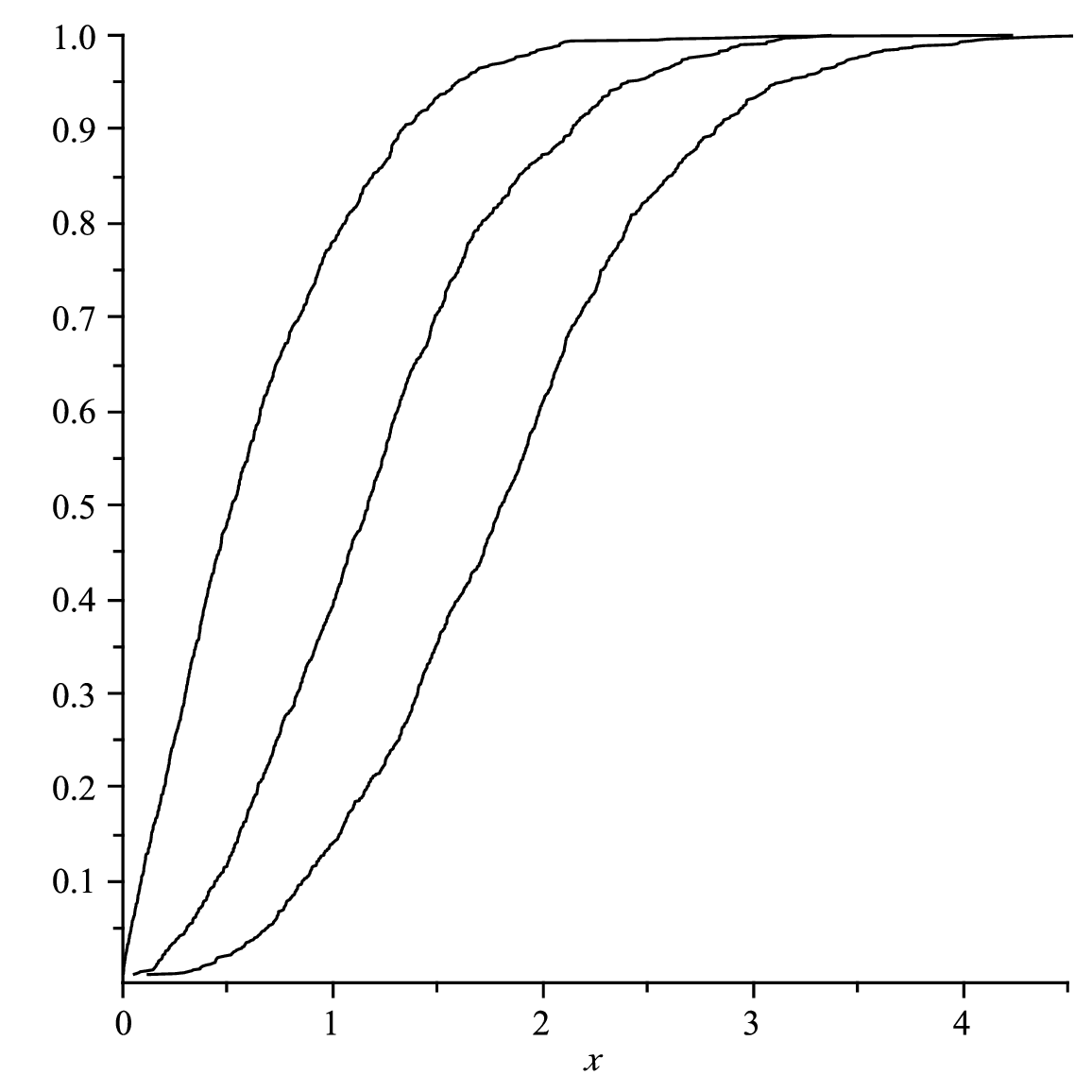}}
\end{center}\caption{Plotted above are the curves $\Pr( \sqrt{n} \sigma_{n-l} (M_{n-l}(\xi)) \le x)$, for $l=0,1,2$ based on data from 1000 randomly generated matrices with $n=100$.  The curves on the left were generated with $\xi$ a random Bernoulli
variable, taking the values $+1$ and $-1$ each with probability $1/2$; and curves on the right were generated with $\xi$ a
gaussian normal random variable.  In both cases, the curves from left to right correspond to the cases $l=0,1,2$, respectively. } 
\label{fig5_2d}
\end{figure}

Figure~\ref{fig5_2d} plots the functions $\Pr( \sqrt{n}\sigma_{n-k} (M_n(\xi)) \le x)$, where $k=0,1,2$, for Bernoulli and gaussian distributions.

%{\bf To Phil: Can you have a picture here, with $l=0,1,2$, say--all in one picture as you already did, but perhaps two different pictures for Bernoulli and Gaussian}

To conclude this section, let us mention two important recent results. 
In   \cite{rv2}, Rudelson and Vershynin obtained 
strong tail estimates for the smallest singular value of 
random rectangular matrices of all possible sizes. 
In \cite{FS} Feldheim and Sodin considered the case when $l$ is large, $l=\Theta (n)$
and  proved universality for the distribution of the least singular value of random matrices with entries having sub-gaussian tails. (In this case the least singular value is
large,  of order $\Theta (\sqrt n)$,  with high probability.)

\subsection{Random matrices with not necessarily identical  entries}

All results hold  if we drop the condition that the entries of $M_{n}$ have the same distribution. 
In the proofs, it is only important that the these entries are all normalized, independent, and their 
$C_{0}$-moment is uniformly bounded. The fact that they are iid copies of $\a$ has not played any important role. The only place where this information was used is in the paragraph following Lemma \ref{proj0} (this allows one to fix the random rows of $B$ to be the first $s$). But we can easily avoid this by conditioning 
on the choice of the rows of $B$.  Thus we have the following extension of Theorem \ref{main-thm}.

\begin{theorem} \label{Thm:notiid} Let $\a_{ij}$ be $\R$- or $\C$-normalized, independent  random variables such that 
 $\E |\a|^{C_0} \le C_{1}$ for some sufficiently large absolute constants $C_0$ and $C_{1}$. 
 Then for all $t>0$, we have
\begin{equation}\label{mna2}
 \P( n \sigma_n( M_n( \a ) )^2 \leq t ) = \int_0^t \frac{1+\sqrt{x}}{2\sqrt{x}} e^{-(x/2 + \sqrt{x})}\ dx + O(n^{-c})
\end{equation}
if $\a_{ij}$ are all  $\R$-normalized, and
$$ \P( n \sigma_n( M_n( \a ) )^2 \leq t ) = \int_0^t e^{-x}\ dx + O(n^{-c})$$
if $\a_{ij}$ are all  $\C$-normalized, where $c>0$ is an absolute constant.
The implied constants in the $O(.)$ notation depend on $\E
|\a|^{C_0}$ but are uniform in $t$.
\end{theorem}

The reader is invited to formalize the extensions of the results in the previous two subsections regarding joint distribution of the least $k$-singular values and rectangular matrices. 

Figure~\ref{fig2BerGau} shows an empirical demonstration of this theorem. 

\begin{figure}
\begin{center}
\scalebox{.5}{\includegraphics{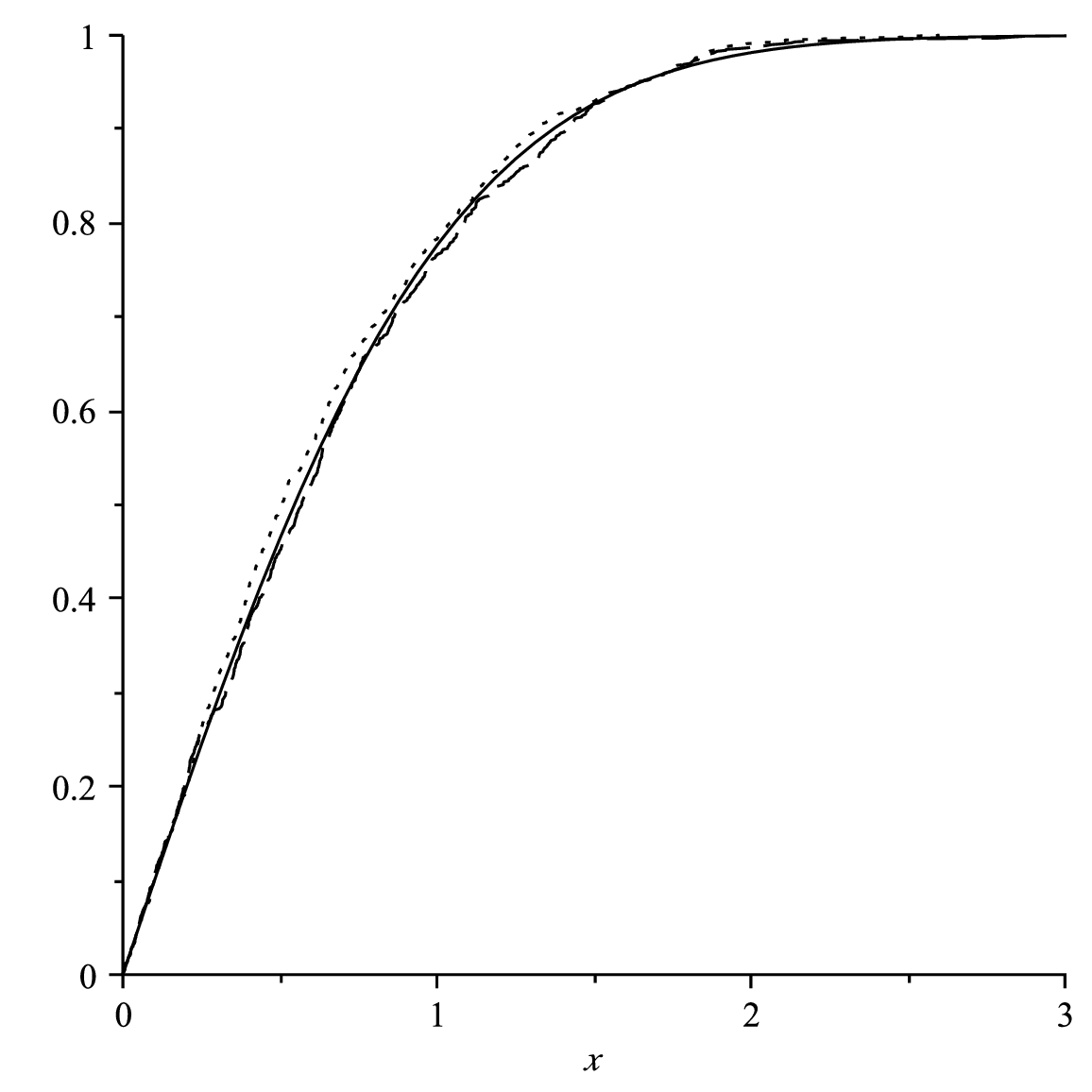}}
\end{center}
\caption{Plotted above is the curve $\Pr( \sqrt{n} \sigma_n(M_n(\xi)) \le x
)$, based on data from 1000 randomly generated matrices with $n=100$. The
dotted curve was generated with $\xi$ a random variable taking the value 0 with probability $1-(1/2)^c$ and taking the
values $+2^{c/2}$ and $-2^{c/2}$ each with probability $(1/2)^{c+1}$, where $c$ is the smallest integer representative of $i+j \mod 3$. The dashed curve was
generated the same way, but with $c=i+j \mod 4$.  The solid curve is a plot of $1-\exp(-x - x^2/2)$, which is the limiting curve predicted by Theorem~\ref{Thm:notiid}.} 
\label{fig2BerGau}

\end{figure}

%{\bf To Phil: Can you have a picture here.}

\subsection{Distribution of the condition number}

Let $M$ be an $n \times n$ matrix, its condition number $\kappa (M)$ is defined as 

$$\kappa (M):= \sigma_{1} (M) / \sigma _{n} (M). $$

Motivated by  an earlier study of von Neuman and Goldstein  \cite{vonG},  Edelman  \cite{edelman}  studied the distribution of $\kappa (M_{n} (\a))$ when $\a$ is gaussian. 
His results can be extended for the general setting in  this paper, thanks to 
Theorem \ref{main-thm} and 
the well known fact that the largest singular value $\sigma _{1} (M_{n} (\a))$ is  concentrated strongly around $2 \sqrt n$.

\begin{lemma} \label{lemma:largestsing} Under the setting of Theorem \ref{main-thm}, we have, with probability $1- \exp(-n^{\Omega(1)} )$, 
$\sigma_{1}(M_{n} (\a)) = (2+o(1)) \sqrt n$.
\end{lemma}

 One can prove this lemma by combining \cite[Theorem 5.8]{BS} with the tools in Appendix \ref{section:MP} or by following the arguments in \cite{AKV, GZ}. See also \cite{rv, rv2} for references.

 \begin{corollary}  \label{corollary:condition} Let $\a_{ij}$ be $\R$- or $\C$-normalized, independent  random variables such that 
 $\E |\a|^{C_0} \le C_{1}$ for some sufficiently large absolute constants $C_0$ and $C_{1}$. 
 Then for all $t>0$, we have
\begin{equation}\label{mna3}
 \P(\frac{1}{2n}  \kappa( M_n( \a ) )  \ge t ) = \int_0^t \frac{1+\sqrt{x}}{2\sqrt{x}} e^{-(x/2 + \sqrt{x})}\ dx + O(n^{-c})
\end{equation}
if $\a_{ij}$ are all  $\R$-normalized, and
$$ \P( \frac{1}{2n} \kappa (M_{n } (\a))  \ge t ) = \int_0^t e^{-x}\ dx + O(n^{-c})$$
if $\a_{ij}$ are all  $\C$-normalized, where $c>0$ is an absolute constant.
The implied constants in the $O(.)$ notation depend on $\E
|\a|^{C_0}$ but are uniform in $t$.
 \end{corollary}

\appendix

\section{Estimating singular values via random sampling}\label{sampling-sec}

\begin{lemma}[Random sampling]\label{rs}  Let $1 \leq s \leq n$ be integers. $A$ be an $n \times n$ real or complex matrix with rows $R_1,\ldots,R_n$.  Let $k_1,\ldots,k_s \in \{1,\ldots,n\}$ be selected independently and uniformly at random, and let $B$ be the $s \times n$ matrix with rows $R_{k_1},\ldots,R_{k_s}$.  Then
$$ \E \| A^* A - \frac{n}{s} B^* B \|_F^2 \leq \frac{n}{s} \sum_{k=1}^n |R_k|^4.$$
\end{lemma}

\begin{proof} Let $a_{ij}$ denote the coefficients of $A$, thus $R_i = (a_{i1},\ldots,a_{in})$.  For $1 \leq i \leq j$, the $ij$ entry of $A^* A - \frac{n}{s} B^* B$ is given by
\begin{equation}\label{joy}
 \sum_{k=1}^n \overline{a_{ki}} a_{kj} - \frac{n}{s} \sum_{l=1}^s \overline{a_{k_l i}} a_{k_l j}.
 \end{equation}
For $l=1,\ldots,s$, the random variables $\overline{a_{k_l i}}
a_{k_l j}$ are iid with mean $\frac{1}{n} \sum_{k=1}^n
\overline{a_{ki}} a_{kj}$ and variance
\begin{equation}\label{vij}
 V_{ij} := \frac{1}{n} \sum_{k=1}^n |a_{ki}|^2 |a_{kj}|^2 - |\frac{1}{n} \sum_{k=1}^n \overline{a_{ki}} a_{kj}|^2,
\end{equation}
and so the random variable \eqref{joy} has mean zero and variance
$\frac{n^2}{s} V_{ij}$.  Summing over $i,j$, we conclude that
$$ \E \| A^* A - \frac{n}{s} B^* B \|_F^2 = \frac{n^2}{s} \sum_{i=1}^n \sum_{j=1}^n V_{ij}.$$
Discarding the second term in \eqref{vij} we conclude
$$ \E \| A^* A - \frac{n}{s} B^* B \|_F^2 \leq \frac{n}{s} \sum_{i=1}^n \sum_{j=1}^n \sum_{k=1}^n |a_{ki}|^2 |a_{kj}|^2.$$
Performing the $i,j$ summations, we obtain the claim.
\end{proof}

Combining this lemma with the Hoefmann-Wielandt theorem (Lemma
\ref{hw}) we conclude that
$$ \E \sum_{i=1}^n (\sigma_i(A)^2 - \frac{n}{s} \sigma_i(B)^2)^2 \leq \frac{n}{s} \sum_{k=1}^n |R_k|^4.$$
Now observe that if $s \leq \sqrt{n}$, then by the standard
birthday paradox computation\footnote{One could remove this
condition $s \leq \sqrt{n}$ by reworking the second moment
computation in Lemma \ref{rs} when $k_1,\ldots,k_s$ are sampled
with replacement (and thus have some slight correlation with each
other), but we will not do so here since in our applications $s$
will be a small power of $n$.}, the rows $k_1,\ldots,k_s$ are
distinct with probability $\Omega(1)$.  Conditioning on this
event, we thus see that if $k_1,\ldots,k_s$ are sampled from
$\{1,\ldots,n\}$ \emph{without} replacement, that
\begin{equation}\label{dest}
 \E \sum_{i=1}^n (\sigma_i(A)^2 - \frac{n}{s} \sigma_i(B)^2)^2 \ll \frac{n}{s} \sum_{k=1}^n |R_k|^4.
\end{equation}
The inequality \eqref{dest} is valid for deterministic matrices
$A$.   If $A$ is a random matrix (with the $k_1,\ldots,k_s$ being
sampled independently of $A$), then we may take expectations of
\eqref{dest} for each instance of $A$ and conclude that
$$
 \E \sum_{i=1}^n (\sigma_i(A)^2 - \frac{n}{s} \sigma_i(B)^2)^2 \ll \frac{n}{s} \E \sum_{k=1}^n |R_k|^4.
$$
Now assume that the random matrix $A$ is \emph{row-exchangeable},
i.e. the distribution of $A$ is stationary with respect to row
permutations $R_i \leftrightarrow R_j$.  For fixed and distinct
$k_1,\ldots,k_s \in \{1,\ldots,n\}$, the distribution of
$\sum_{i=1}^n (\sigma_i(A)^2 - \frac{n}{s} \sigma_i(B)^2)^2$ is
independent of the choice of $k_1,\ldots,k_s$.  Crudely bounding
$|R_k|$ by $\max_{1 \leq k \leq n} |R_k|$, we conclude

\begin{corollary}[Sampling of row-exchangeable matrices]\label{srm}  Let $s, n$ be integers with $1 \leq s \leq \sqrt{n}$, let $A$ be a row-exchangeable random $n \times n$ matrix with rows $R_1,\ldots,R_n$, and let $B$ be the $s \times n$ matrix with rows $R_1,\ldots,R_s$.  Then
$$ \E \sum_{i=1}^n (\sigma_i(A)^2 - \frac{n}{s} \sigma_i(B)^2)^2 \ll \frac{n^2}{s} \E (\max_{1 \leq k \leq n} |R_k|)^4.$$
\end{corollary}

\begin{remark} One could use row-exchangeability here to instead simplify $\E \sum_{k=1}^n |R_k|^4$ as $n \E |R_1|^4$, but it will turn out that this will not be useful for us, as we will need to truncate $\max_{1 \leq k \leq n} |R_k|$ in any event in order to assure that $\E |R_1|^4$ is bounded.
\end{remark}

\begin{remark} In our applications, $A$ is going to be the inverse of $M_n(\a)$
(after conditioning away some bad but rare events in which
$M_n(\a)$ is close to degenerate),
 and the summand of most interest on the left-hand side is the $i=1$ summand.
  We expect $\sigma_1(A)$ to be of size about $\sqrt{n}$, and we also expect the $|R_k|$ to have size about $O(1)$ (thanks to Lemma \ref{dvh} and basic heuristics about the distance between a random vector and a random hyperplane).  So, as soon as $s$ is a non-trivial power of $n$, we expect
Corollary \ref{srm} to yield an approximation of the form
\eqref{sigman}.
\end{remark}

\section{Linear algebra} \label{section:algebra}

In this appendix we collect various basic facts from linear algebra which we will rely upon in this paper.
 We begin with the Hoeffman-Wielandt theorem, which controls the variation of singular values of a matrix using the Frobenius norm.

\begin{lemma}[Hoeffman-Wielandt theorem]\label{hw}  For any two $n \times n$ self-adjoint matrices $M, M'$, we have
$$ \sum_{i=1}^n (\sigma_i(M) - \sigma_i(M'))^2 \leq \| M - M' \|_F^2.$$
As a corollary, for any two $n \times n$ matrices $A, B$, we have
$$ \sum_{i=1}^n (\sigma_i(A)^2 - \sigma_i(B)^2)^2 \leq \| A^* A - B^* B \|_F^2.$$
\end{lemma}

\begin{proof} It suffices to prove the first inequality, which we rewrite as
$$ (\sum_{i=1}^n (\sigma_i(M+N) - \sigma_i(M))^2)^{1/2} \leq \| N \|_F$$
for self-adjoint $M, N$.  By the fundamental theorem of calculus (or a compactness argument) and the triangle inequality in $l^2$, it suffices to show the infinitesimal version
$$ (\sum_{i=1}^n (\sigma_i(M+\eps N) - \sigma_i(M))^2)^{1/2} \leq \eps \| N \|_F + o(\eps)$$
of this inequality for fixed $M,N$ and sufficiently small $\eps > 0$.  But if we diagonalise $M$, the left-hand side can be computed (up to errors of $o(\eps)$) as $\eps$ times the $l^2$ norm of the diagonal of $N$, and the claim follows.
\end{proof}

Using the minimax characterizations
$$ \sigma_j(A) = \sup_{V: \dim(V) = j} \inf_{v \in V: |v| = 1} |Av|$$
of the singular values, where $V$ ranges over $j$-dimensional
subspaces of $\C^n$, one obtains Weyl's bound

\begin{equation}\label{lips}
\| \sigma_j(A) - \sigma_j(B) \| \leq \| A - B \|_{op}
\end{equation}
for any $m \times n$ matrices $A, B$.  Another easy corollary of the minimax characterization is

\begin{lemma}[Cauchy interlacing law]\label{cauchy}
Let $A$ be an $m \times n$ matrix, and let $A'$ be a $r \times n$ matrix formed by taking $r$ of the $m$ rows of $A$.  Then
$$ \sigma_j(A) \geq \sigma_j(A') \geq \sigma_{j+m-r}(A)$$
for all $1 \leq j \leq \min(r,n)$, with the convention that $\sigma_j(A)=0$ for $j > \min(m,n)$.
\end{lemma}

Next, we prove the elementary but crucial projection lemma (Lemma
\ref{proj0}) that relates a random sample of an inverse matrix, to
a randomly projected version of the original matrix. For the
reader's convenience, we restate this lemma.

\begin{lemma}[Projection lemma]\label{proj}  Let $1 \leq s \leq n$ be integers, let $F$ be the real or complex field, let $A$ be an $n \times n$ $F$-valued invertible matrix with columns $X_1,\ldots,X_n$, and let $R_1,\ldots,R_n$ denote the rows of $A^{-1}$.  Let $B$ be the $s \times n$ matrix with rows $R_1,\ldots,R_s$.  Let $V$ be the $s$-dimensional subspace of $F^n$ formed as the orthogonal complement of the span of $X_{s+1},\ldots,X_n$, which we identify with $F^s$ via an orthonormal basis, and let $\pi: F^n \to F^s$ be the orthogonal projection to $V \equiv F^s$.  Let $M$ be the $s \times s$ matrix with columns $\pi(X_1),\ldots,\pi(X_s)$.  Then $M$ is invertible, and we have
$$ B B^* = M^{-1} (M^{-1})^*.$$
In particular, we have
$$ \sigma_j( B ) = \sigma_{s-j+1}( M )^{-1}$$
for all $1 \leq j \leq s$.
\end{lemma}

\begin{proof} By construction, we have $R_i \cdot X_j = \delta_{ij}$ for all $1 \leq i,j \leq n$. In particular, $R_1,\ldots,R_s$ lie in $V$, and
$$ \pi(R_i) \cdot \pi(X_j) = \delta_{ij}$$
for $1 \leq i,j \leq s$.  Thus $M$ is invertible, and the rows of $M^{-1}$ are given by $\pi(R_1),\ldots,\pi(R_s)$.  Thus the $ij$ entry of $(M^{-1}) (M^{-1})^*$ is given by $\pi(R_i) \cdot \pi(R_j)$, while the $ij$ entry of $BB^*$ is given by $R_i \cdot R_j$.  Since $R_1,\ldots,R_s$ lie in $V$, and $\pi$ is an isometry on $V$, the claim follows.
\end{proof}

\section{Correlation between distances}
\label{section:correlation}

\begin{lemma}[Relationship between different distances]\label{Relation2} Let $n \geq 1$, let $F$ be the real or complex field, let $A$ be an $n \times n$ $F$-valued invertible matrix with columns $X_1,\ldots,X_n$, and let $d_i := \dist(X_i,V_i)$ denote the distance from $X_i$ to the hyperplane $V_i$ spanned by $X_1,\ldots,X_{i-1},X_{i+1},\ldots,X_n$.  Let $1 \leq L < j \leq n$, let $V_{L,j}$ denote the orthogonal complement of the span of $X_{L+1},\ldots,X_{j-1},X_{j+1},\ldots,X_n$, and let $\pi_{L,j}: F^n \to V_{L,j}$ denote the orthogonal projection onto $V_{L,j}$.  Then
$$ d_j \geq \frac{ | \pi_{L,j}(X_j) | }{ 1 + \sum_{i=1}^L \frac{|\pi_{L,j}(X_i)|}{d_i} }.$$
\end{lemma}

\begin{remark} In practice, we will be able to get good upper and lower bounds on $|\pi_{L,j}(X_i)|$.
The above lemma asserts that as long as the $d_1,\ldots,d_L$ are bounded away from zero,
all the other $d_j$ will also be bounded away from zero.
\end{remark}

\begin{proof}  By relabeling we may take $j = L+1$.  For $1 \leq i \leq L+1$, write $Y_i := \pi_{L,L+1}(X_i) \in V_{L,j}$.  By applying $\pi_{L,L+1}$ to both $X_i$ and $V_i$, we see that $d_i$ is the distance from $Y_i$ to $Y_1,\ldots,Y_{i-1},Y_{i+1},\ldots,Y_{L+1}$ for any $1 \leq i \leq L+1$.  Since $V_{L,j}$ is isomorphic to $F^{L+1}$, we see (on replacing $n$ with $L+1$, and $X_i$ with $Y_i \in V_{L,j} \equiv F^{L+1}$) that we may assume that $L+1=n$, thus $\pi_{L,j}$ is trivial and our task is now to show that
$$ d_n \geq \frac{ | X_n | }{ 1 + \sum_{i=1}^{n-1} \frac{|X_i|}{d_i} }.$$
Let $R_1,\ldots,R_n$ be the rows of $A^{-1}$.  By Corollary \ref{dvh}, $d_i = 1/|R_i|$, so our task is now to show that
$$ |R_n| \leq \frac{1}{|X_n|} + \sum_{i=1}^{n-1} \frac{|X_i|}{|X_n|} |R_i|.$$
On the other hand, observe (since $R_1,\ldots,R_n$ is a dual basis to $X_1,\ldots,X_n$) that the vector
$$ \frac{X_n}{|X_n|^2} - \sum_{i=1}^{n-1} \frac{X_i \cdot X_n}{|X_n|^2} R_i$$
is orthogonal to $X_1,\ldots,X_{n-1}$ and has an inner product of $1$ with $X_n$, and
thus must equal $R_n$.  By the triangle inequality and Cauchy-Schwarz we thus obtain the claim.
\end{proof}

\section{Berry-Esseen type results} \label{section:be} 

Let us first state the classical Berry-Ess\'een central limit
theorem (see e.g. \cite{stroock}):

\begin{proposition}[Berry-Ess\'een theorem]\label{berry}  Let $v_1,\ldots,v_n \in \R$ be real numbers with
\begin{equation}\label{vii}
v_1^2 + \ldots + v_n^2 = 1,
\end{equation}
and let $\a$ be a $\R$-normalized random variable with finite third moment $\E |\a|^3 < \infty$.  Let $S \in \R$ denote the random variable
$$ S = v_1 \a_1 + \ldots + v_n \a_n$$
where $\a_1,\ldots,\a_n$ are iid copies of $\a$. Then for any $t \in \R$ we have
$$ \P( S \leq t ) =  \P( \g_\R \leq t ) + O( \sum_{j=1}^n |v_j|^3 ),$$
where the implied constant depends on the third moment $\E |\a|^3$ of $\a$.  In particular, by \eqref{vii} we have
$$ \P( S \leq t ) =  \P( \g_\R \leq t ) + O( \max_{1 \leq j \leq n} |v_j| ).$$
\end{proposition}

Next, we prove the key Proposition \ref{berry-esseen-frame0}, which
can be seen as a high-dimensional extension of Proposition
\ref{berry}. Let us restate this proposition for the reader's
convenience.

\begin{proposition}[Berry-Ess\'een-type central limit theorem for frames]\label{berry-esseen-frame} Let $1 \leq N \leq n$, let $F$ be the real or complex field, and let $\a$ be $F$-normalized and have finite third moment $\E |\a|^3 < \infty$.  Let $v_1,\ldots,v_n \in F^N$ be a \emph{normalized tight frame} for $F^N$, or in other words
\begin{equation}\label{vivi2}
v_1 v_1^* + \ldots + v_n v_n^* = I_N,
\end{equation}
where $I_N$ is the identity matrix on $F^N$.  Let $S \in F^N$ denote the random variable
$$ S = \a_1 v_1 + \ldots + \a_n v_n,$$
where $\a_1,\ldots,\a_n$ are iid copies of $\a$.  Similarly, let $G := (\g_{F,1},\ldots,\g_{F,N}) \in F^N$
be formed from $N$ iid copies of $\g_F$.  Then for any measurable set $\Omega \subset F^N$ and any $\eps > 0$, one has
\begin{align*}
 \P( G \in \Omega \backslash \partial_\eps \Omega ) - & O(N^{5/2} \eps^{-3} (\max_{1 \leq j \leq n} |v_j|))
\leq
\P( S \in \Omega ) \\ &\leq \P( G \in \Omega \cup \partial_\eps \Omega ) + O(N^{5/2} \eps^{-3} (\max_{1 \leq j \leq n} |v_j|)), \end{align*}
where
$$ \partial_\eps \Omega := \{ x \in F^N: \dist_\infty( x, \partial \Omega ) \leq \eps \},$$
$\partial \Omega$ is the topological boundary of $\Omega$, and
and $\dist_\infty$ is the distance using the $l^\infty$ metric on $F^N$.  The implied constant depends on the third moment $\E |\a|^3$ of $\a$.
\end{proposition}

\begin{proof} We shall just prove the upper bound
$$ \P( S \in \Omega ) \leq \P( G \in \Omega \cup \partial_\eps \Omega ) + O(N^{5/2} \eps^{-3} (\max_{1 \leq j \leq n} |v_j|)),$$
as the lower bound then follows by applying the claim to the complement of $\Omega$.

Let $\psi: F \to \R^+$ be a bump function supported on the unit ball $\{x \in F: |x| \leq 1 \}$ of total mass $\int_F \psi = 1$, let $\Psi_{\eps,N}: F^N \to \R^+$ be the approximation to the identity
$$ \Psi_{\eps,N}(x_1,\ldots,x_N) := \prod_{i=1}^N \frac{1}{\eps} \psi( \frac{x_i}{\eps} ),$$
and let $f: F^n \to \R^+$ be the convolution
\begin{equation}\label{third}
 f(x) = \int_{F^N} \Psi_{\eps,N}(y) 1_\Omega( x-y)\ dy
 \end{equation}
where $1_\Omega$ is the indicator function of $\Omega$.  Observe that $f$ equals $1$ on $\Omega \backslash \partial_\eps \Omega$, vanishes outside of $\Omega \cup \partial_\eps \Omega$, and is smoothly varying between $0$ and $1$ on $\partial_\eps \Omega$.  Thus it will suffice to show that
$$ |\E( f( S ) ) - \E( f( G ) )| \ll N^{5/2} \eps^{-3} (\max_{1 \leq j \leq n} |v_j|).$$

We now use a Lindeberg replacement trick (cf. \cite{lindeberg, PR}).
 Let $\g'_{F,1},\ldots,\g'_{F,n}$ be $n$ iid copies of $\g_F$.  From \eqref{vivi2} and the fact that the distribution of centered
  gaussian random variables are completely determined by their covariance matrix, we see that
$$ \g'_{F,1} v_1 + \ldots + \g'_{F,n} v_n \equiv G.$$
Thus if we define the random variables
$$ S_j := \a_1 v_1 + \ldots + \a_j v_j + \g'_{F,j+1} v_j + \ldots + \g'_{F,n} v_n \in F^N,$$
we have the telescoping triangle inequality
\begin{equation}\label{fsag}
 |\E(f(S)) - \E( f(G) )| \leq \sum_{j=1}^{n} |\E f( S_{j} ) - \E f( S_{j-1} )|,
\end{equation}
For each $1 \leq j \leq n$, we may write
$$ S_j = S'_j + \a_j v_j; S_{j-1} = S'_j + \g'_{F,j} v_j$$
where
$$ S'_j := \a_1 v_1 + \ldots + \a_{j-1} v_{j-1} + \g'_{F,j+1} v_j + \ldots + \g'_{F,n} v_n.$$
By Taylor's theorem with remainder we thus have
\begin{equation}\label{fa}
 f(S_j) = f(S'_j) + \a_j (v_j \cdot \nabla) f(S'_j) + \frac{1}{2} \a_j^2 (v_j \cdot \nabla)^2 f(S'_j) + O( |\a_j|^3 \sup_{x \in F^n} |(v_j \cdot \nabla)^3 f(x)| )
 \end{equation}
and
\begin{equation}\label{fb}
 f(S_{j-1}) = f(S'_j) + \g_{F,j} (v_j \cdot \nabla) f(S'_j) + \frac{1}{2} \g_{F,j}^2 (v_j \cdot \nabla)^2 f(S'_j) + O( |\g_{F,j}|^3 \sup_{x \in F^n} |(v_j \cdot \nabla)^3 f(x)| )
\end{equation}
in the real case $F=\R$, with a similar formula in the complex case $F=\C$ obtainable by decomposing $\a_j, \g_{F,j}$ into real and imaginary parts, which we will omit here.  A computation using \eqref{third} and the Leibnitz rule reveals that all third partial derivatives of $f$ have magnitude $O( \eps^{-3} )$, and so by Cauchy-Schwarz we have
$$ \sup_{x \in F^n} |(v_j \cdot \nabla)^3 f(x)| \ll |v_j|^3 N^{3/2} \eps^{-3}.$$
Observe that $\a_j$, $\g_{F,j}$ are independent of $S'_j$, and have the same mean and variance (in the complex case $F=\C$, we would use the covariance matrix of the real and imaginary parts rather than just the variance).  Subtracting \eqref{fa} from \eqref{fb} and taking expectations (using the bounded third moment hypothesis on $\a$) we conclude that
$$ | \E( f(S_j) ) - \E( f( S_{j-1} ) )| \ll |v_j|^3 N^{3/2} \eps^{-3} $$
and thus by \eqref{fsag}
$$ |\E(f(S)) - \E( f(G) )| \ll N^{3/2} \eps^{-3} \sum_{j=1}^n |v_j|^3.$$
On the other hand, by taking traces of \eqref{vivi2} we have
$$ \sum_{j=1}^n |v_j|^2 = N$$
and the claim follows.
\end{proof}

\begin{remark}  In our applications, $N$ and $1/\eps$ will be a very small power of $n$.
 The theorem then becomes non-trivial as soon as one obtains an  upper bound of the
  form $n^{-c}$ on the vectors $|v_i|$ for some absolute constant $c > 0$.
\end{remark}

\begin{remark}  Suppose $\a$ was now an arbitrary complex random variable of zero mean and finite variance.  If the covariance matrix of $\Re \a$ and $\Im \a$ was a scalar multiple of the identity, then one is essentially in the $\C$-normalized case and Proposition \ref{berry-esseen-frame} applies.  If instead the covariance matrix was degenerate and the $v_1,\ldots,v_n$ had purely real coefficients, then one is in the $\R$-normalized case and Proposition \ref{berry-esseen-frame} again applies.  But the situation is more complicated when the covariance matrix has two distinct non-zero eigenvalues, and depends to the extent to which the phases of $v_1,\ldots,v_n$ are aligned.  The question of what happens to the least singular value of $M_n(\a)$ in this case seems to be an interesting one (presumably, there is no alignment of phases and one should essentially revert to the $\C$-normalized case), but we will not pursue it here.
\end{remark}

\section{Concentration} \label{section:concentration}

Let us state a powerful theorem of Talagrand. Let $\BOmega_{i} $ be probability spaces equipped with 
measures $\mu_{i}$, $i=1, \dots, n$ and $\BOmega= \BOmega_{1} \times \dots \times \BOmega_{n} $ be the
product space equipped with the product measure $\mu=\mu_{1} \times \dots \times \mu_{n} $. 
A point $x \in \BOmega$ has coordinates $(x_{1}, \dots, x_{n} )$.

For a unit vector $a= (a_{1}, \dots, a_{n})$ with non-negative coordinates and $x,y \in \BOmega$, define 

$$d_{a} (x,y) := \sum_{i=1}^{n}  a_{i } \BI_{x_{i} \neq y_{i}}. $$

For a subset $A \subset \BOmega$ and $x \in \BOmega$, let 

$$d_{a} (x,A):= \inf _{y \in A} d_{a } (x,y) $$ and 

$$D(x,A):= \sup_{a, |a|=1} d_{a } (x, A). $$

In practice, the following (equivalent) definition of $D(x,A)$ is useful. Define 

$$U(x,A):= \{s \in \{0,1\}^{n} , \,\, \hbox {there is} \, y \in A \,\, \hbox{such that} \,\, y_{i}=x_{i} \,\,\hbox 
{if} \,\, s_{i}=0 \}. $$

It is not too hard (see \cite[Chapter 4]{ledoux}, \cite{tal}) to prove that    

\begin{equation} \label{eqn:newdef} 
D(x,A) := \dist (0, \conv U(x,A) ) = \inf _{s \in \conv U(x,A)} \|s\|,
\end{equation}  where $\dist$, as usual, denotes the Euclidean distance.

\begin{theorem} [Talagrand's inequality]\label{theorem:T0} \cite{tal, ledoux}
For any measurable set $A \subset \Omega$ and $r \ge 0$

$$\mu (x, D(x,A)\ge r) \mu(A) \le \exp(-r^{2}/4 ). $$

\end{theorem}

We will use heavily  the following
corollary of Theorem \ref{theorem:T0}. 

\begin{theorem} \label{theorem:T}
Let $\D$ be the unit disk  $\{z\in \C, |z| \le 1 \}$. For every
product probability $\mu$ on $\D^n$, every convex $1$-Lipschitz
function $F: \C^n \to \R$, and every $r \ge 0$,
$$\mu (|F- M(F)| \ge r) \le 4 \exp(-r^2/16), $$
where $M(F)$ denotes the median of $F$.
\end{theorem}

This corollary is the
complex version of \cite[Corollary 4.10]{ledoux} and can be
obtained by slightly modifying its proof. We provide the (simple)
details for the sake of completeness.

\begin{proof}(Proof of Theorem \ref{theorem:T})
Let $A$ be a subset of $\D^{n}$ and $x$ be a point in $\D^{n}$. By \eqref{eqn:newdef} there is a point $s \in \conv U(x,A)$ such that 

$$D(x,A) = \|s\|. $$

By Caratheodory's theorem there are $s^{1}, \dots, s^{k}  \in \cap U(x,A)$ with $k \le n+1$ such that $s$ is an affine combination of the $s^{j} $. In other words, there are positive numbers $c_{1}, \dots, c_{k}$ such that $\sum_{j=1} ^{k } c_{j} =1 $ and 

$$s= c_{1 } s^{1} + \dots + c_{k} s^{k} . $$

Let $s_{i} $ ($s^{j} _{i} $) be the $i$th coordinate of $s$ ($s^{j} $), respectively. Let $y^{j} $ be the point in $A$ that corresponds to $s^{j} $ (with respect to the definition of $U(x,A)$). We have 

$$D(x,A)=\|s\| = \sqrt {\sum_{i=1} ^{n } (\sum_{j=1} ^{k } c_{j} s^{j} _{i} )^{2} }
\ge \sqrt {\sum_{i=1} ^{n } (\sum_{j=1} ^{k } c_{j} \BI_{x_{i}\neq y^{j}_{i}} )^{2} }. $$

Since $|x_{i}- y^{j} _{i} |\le 2 $ (here is the only place where we use the fact that $\D$ has bounded perimeter), it follows

$$2 D(x,A)= 2 \|s\| \ge \sqrt {\sum_{i=1} ^{n } (\sum_{j=1} ^{k } c_{j} (x_{i}-y^{j}_{i})^{2} }
= \|x- \sum_{j=1} ^{k} c_{j} y^{j} \|. $$

Notice that the right hand side is at least the Euclidean distance from $x$ to the convex  hull of $A$, so we have 

\begin{equation} \label{eqn:convex} D(x,A) \ge \frac{1}{2} \dist (x, \conv A). \end{equation}

Let $F$ be a function as in the statement of the theorem. Let $A:= \{x| F(x) \le M(F)\}$. By the definition of median $\mu (A) \ge 1/2$. Furthermore, as $F$ is convex, so is $A$. On the other hand, if 
$F(x) \ge M(F)+r$, then by the Lipschitz property 
$\dist (x, A) \ge r$, which, via \eqref{eqn:convex}, implies that $D(x,A) \ge r/2$. By Theorem \ref{theorem:T0}, we have 

$$\mu (x, F(x) \ge M(F)+r ) \mu (A) \le \exp (-r^{2}/16) , $$ which implies 

$$\P( F(x) -M(F) \ge r) \le 2 \exp(-r^{2} /16). $$ For the lower tail, set $A:=\{x, F(x) \le M(F)-r\}$ and argue similarly. 
\end{proof}

An easy change of variables reveals the following generalization
of this inequality: if $\mu$ is supported on a dilate $K \cdot \D^n$ of
the unit disk for some $K > 0$, rather than $\D^n$ itself, then for every $r > 0$ we have
\begin{equation}\label{mood}
\mu (|F- M(F)| \ge r) \le 4 \exp(-r^2/16K^2).
\end{equation}

Theorem \ref{theorem:T} shows concentration around the median. In applications, it is usually more useful to have concentration around the mean. This can be done via the following lemma, which shows that concentration around the median implies that the mean and the median are close.

\begin{lemma} \label{lemma:MM} 
Let $X$ be a random variable such that for any $r \ge 0$

$$\P(|X- M(X)| \ge r) \le 4 \exp(-r^{2} ) . $$ Then 

$$|\E(X)- M(X) | \le 100. $$
\end{lemma} 

The bound 100 is \emph{ad hoc} and can be replaced by a much smaller constant. 

\begin{proof} Set $M:=M(X)$ and let $F(x)$ be the distribution function of $X$. We have 

$$\E(X)= \sum_{i} \int_{M-i} ^{M-i+1} x \dd F(x) \le M + 4 \sum_{i} |i| e^{-i^{2}}  \le M + 100. $$

The lower bound can be proved similarly. 
\end{proof}

Now we use Theorem \ref{theorem:T} to prove Lemma
\ref{corup0}. Let us first  restate this lemma.

\begin{lemma}[Concentration estimate]\label{corup} Let $n \geq d \geq 1$ be integers,
let $K \geq 1$, let $F$ be the real or complex field, let $\a$ be $F$-normalised with $|\a| \leq K$ almost surely, and let $V$ be a subspace of $F^n$ of dimension $d$.  Let $X := (\a_1,\ldots,\a_n)$, where $\a_1,\ldots,\a_n$ are iid copies of $\a$.  Suppose that $d > CK^2$ for some sufficiently large absolute constant $C$.  Then
$$ \P( |\dist(X,V) - \sqrt{d}| \geq \sqrt{d}/2 ) \ll \exp( - \Omega( d / K^2 ) ).$$
\end{lemma}

\begin{proof} The map $X \mapsto \dist(X,V)$ is clearly convex and $1$-Lipschitz.  Applying \eqref{mood} we conclude that
\begin{equation}\label{po}
 \P( |\dist(X,V) - M( \dist(X,V) )| \geq t ) \ll \exp( - \Omega( t^2 / K^2 ) )
\end{equation}
for any $t > 0$.  On the other hand, if $\pi: F^n \to V$ is the orthogonal projection to $V$, then (by the $F$-normalization of $\a$) we have
\begin{align*}
\E \dist(X,V)^2 &= \E X^* \pi^* \pi X \\
&= \tr( \pi^* \pi ) \\
&= \dim(V) \\
&= d.
\end{align*}
By Chebyshev's inequality, we conclude that the median $M( \dist(X,V) )$ of $\dist(X,V)$ is at most $2 \sqrt{d}$.
 A simple calculation reveals that the fact
$$ |\dist(X,V)^2 - M( \dist(X,V) )^2| \geq t  $$ implies $$ |\dist(X,V) - M( \dist(X,V) )| \gg \min( t/\sqrt{d}, \sqrt{t} ).$$
(This is easiest to see by working in the contrapositive.)  Thus,
$$ \P( |\dist(X,V)^2 - M( \dist(X,V) )^2| \geq t ) \ll \exp( - \Omega( t / K^2 ) ) + \exp( - \Omega( t^2 / dK^2 ) ).$$
Integrating this, we see that
$$ d = \E \dist(X,V)^2 = M(\dist(X,V))^2 + O( K^2 ) + O( \sqrt{d} K );$$
if $d > CK^2$ for a sufficiently large $C$, we conclude that
$$ 0.9 \sqrt{d} \leq M(\dist(X,V)) \leq 1.1 \sqrt{d}$$
and the claim follows from \eqref{po}.
\end{proof}

\begin{remark} One can obtain a more precise result by also computing the variance of
$\dist(X,V)^2$; see \cite[Lemma 2.2]{tv-det} for this computation in the model case of the Bernoulli random variable.
\end{remark}

\section{Random matrices have many small singular values}\label{section:MP}

%With additional moment assumptions on $\a$, one has a quantitative rate of convergence.  For instance, we have the following result of Bai, Miao, and Tsay \cite{bmt}:

In this section, we prove Lemma \ref{bai-cor0}, which we now restate.

\begin{lemma}[Many small singular values]\label{bai-cor}  Let $n \geq 1$, and let $\a$ be $\R$-normalized or $\C$-normalized,
 such that $|\a| \leq n^\eps$ almost surely for some sufficiently small absolute constant $\eps > 0$.
 Then there are positive constants $c, c'$ such that with probability 
 $1- \exp(-n^{\Omega (1)} )$, $M_{n} (\a)$ has at least 
 $c' n^{1-c}$ singular values in the interval $[0, n^{1/2-c}]$. 
\end{lemma}

The first ingredient of the proof is the following result on the rate of convergence to Marchenko-Pastur law. 

\begin{theorem}[Rate of convergence  to Marchenko-Pastur law]\label{theorem:rate}\cite{BS} 
There are positive constants $c_{1}, c_{2} $ such that the following holds. 
 Let $\a$ be $\R$- or $\C$-normalized random variable with bounded $c_{1}$-moment. 
Then for any $t \ge 0$, 

$$ \P\left( \sup_{t \geq 0} |\frac{1}{n} | \{ 1 \leq i \leq n: \frac{1}{n} \sigma_i( M_n( \a ) )^2 \leq t \}| - \frac{1}{2\pi} \int_0^{\min(t,4)} \sqrt{\frac{4}{x}-1}\ dx| \ge n^{-c_{2}} \right) = o(1). $$

In other words, the ESD of the Wishart matrix $\frac{1}{n} M_{n} ^{\ast } M_{n} $ converges to Marchenko-Pastur law with rate $n^{-c_{2}}$ almost surely. 

\end{theorem}

The values for $c_{1}, c_{2} $ are quite reasonable. In \cite[Theorem 8.29]{BS}, it is shown that one can set  $c_{1}=6$ and $c_{2} = 1/6 -\epsilon$, for any fixed $\epsilon >0$.  A more recent result 
\cite[Theorem 1.2]{GT} claimed that one can set $c_{1} =4$ and $c_{2}=1/2 -\epsilon$.

This theorem, however, does not imply the desired claim, as it only shows that $M_{n}$ has many 
small singular values with probability $1-o(1)$, while we need the failure probability to be exponentially small.  

For a constant $1/2 >c>0$, let  $Y_{c}$ be the number of singular values in the interval $[0, n^{1/2-c}]$.
From Theorem \ref{theorem:rate}, we  can at least conclude that

\begin{equation} \label{eqn:largemean} \E(Y_{c})= \Omega (n^{1-c}). \end{equation}

We are going to show that for some sufficiently small positive constant $\alpha$

\begin{equation} \label{eqn:exponential} \P (Y_{c} \le \alpha n^{1-c}) \le \exp (-n^{\Omega(1)} ). \end{equation}

One can achieve this goal  by following, with few minor modifications, the powerful
 approach introduced by Guionnet and Zeitouni 
 in \cite{GZ}.  We present the details for the sake of completeness and the readers' convenience. 
 
 Consider a random hermitian  matrices $W_{N}$ with 
 independent entries $w_{ij},  1\le i \le j \le N $ with support in a compact region $S$ with diameter 
 $K$. Let $f$ be a real, convex , $1$-Lipschitz function and define 
 
 $$Z := \sum_{i=1} ^{N} f(\lambda_{i} ) $$ where $\lambda_{i} $ are the eigenvalues of $\frac{1}{\sqrt N}W_{N}$. 
 We are a going to view $Z$ as the function of the variables  $w_{ij}, 1\le i \le j \le N$.  
 The main difference between the current setting and  that of  \cite{GZ} is that 
here we do not require  the real and imaginary parts of $w_{ij} $ be independent.

We use the following two lemmas from \cite{GZ}.

\begin{lemma}  \label{lemma:GZ1}  $Z$ is a convex function. 
\end{lemma}

This lemma is a consequence of  \cite[Lemma 1.2(a)]{GZ}. 

\begin{lemma} \label{lemma:GZ2} $Z$ is $\sqrt 2$-Lipschitz. 
\end{lemma}

This lemma can be derived from \cite[Lemma 1.2(b)]{GZ}. We give here a short, different, proof. 

\begin{proof} (Proof of Lemma \ref{lemma:GZ2}) 
Consider two matrices $W$ and $W'$ with entries $w_{ij} $ and $w_{ij}'$  and eigenvalues 
$\lambda_{i} $ and $\lambda_{i}'$ (in decreasing order), respectively. By Lemma \ref{hw}, 

$$\sum_{i=1} ^{N} |\lambda_{i} -\lambda_{i}'| ^{2} | \le  \frac{1}{N} \| W-W' \| _{F} ^{2} \le 
\frac{2}{N}  \sum_{1 \le i \le j \le N} |w_{ij} -w_{ij}'|^{2} . $$

On the other hand, by Cauchy-Schwarz and the fact that $f$ is $1$-Lipschitz,

$$|Z-Z'|^{2 } \le N \sum_{i=1}^{N} |f(\lambda_{i} )-f(\lambda_{i}|^{2} \le N |\lambda_{i}-\lambda_{i}'|^{2} . $$

It follows that 

$$|Z-Z'| \le \sqrt 2 \big(\sum_{1 \le i \le j \le N} |w_{ij} -w_{ij}'|^{2} \big) ^{1/2} $$ completing the proof. 
 \end{proof}

 By Theorem \ref{theorem:T}, we obtain the following theorem, which is an extension  of 
 \cite[Theorem 1.3]{GZ} to the case where the real and imaginary parts of $w_{ij}$ are not necessarily independent.

\begin{theorem} \label{theorem:GZ}  Let $W_{N}, f, Z$ be as above. 
Then there is a constant $c>0$ such that  for any $T \ge 0$

$$\P ( |Z - M(Z) | \ge T ) \ge  4 \exp( -c T^{2}/K^{2} ),$$ 

\noindent where $K$ is the bound on the absolute values of the entries of 
$W_{N}$.

\end{theorem}

It will be better for us to replace $M(Z)$ by $\E(Z)$. By Lemma \ref{lemma:MM} and rescaling, 

$$|M(Z)-\E(Z)|= O(K^{2}). $$

Thus, if $K^{2} =o(T)$, then (by adjusting $c$ if necessary) we have 

\begin{equation} \label{eqn:GZmean} \P ( |Z - \E(Z) | \ge T ) \ge  4 \exp( -c T^{2}/K^{2} ). \end{equation}

Recall that we are considering the singular values of a (random) non-hermitian matrix $M_{n}$ instead of the eigenvalues of a hermitian matrix $W_{N}$. This problem can be easily dealt with by the standard 
trick of defining $N:=2n$ and 

$$W_{N}:=  \left( \begin{matrix}  0  & M_{n} \\ M_{n} ^{\ast} & 0 \end{matrix} \right ) . $$

It is well known that if the singular values of $M$ are $\sigma_{1} , \dots, \sigma_{n} $, then the eigenvalues of $W$ are $\pm \sigma_{1}, \dots, \pm \sigma_{n}$. 
The number of singular values of $\frac{1}{\sqrt n}M$ in $[0, n^{-c}]$ is thus half of the number of eigenvalues of $\frac{1}{\sqrt n}W$ in $I:= [-n^{-c}, n^{-c}]$. In order to estimate the last quantity, it is natural to define 

$$Z :=\sum _{i=1}^{N} \chi_{I} (\lambda_{i} )$$ where $\chi_{I} $ is the indicator function of $I$ and $\lambda_{i} $ are the eigenvalues of $W_{N}$. This function is, however, not convex and Lipschitz. On the other hand, we can easily overcome this problem by constructing two real functions $f_{1},f_{2}
 $ such that

 \begin{itemize} 
 
 \item $f_{j}$ are symmetric, convex and $\frac{C}{|I|}$-Lipschitz, for some sufficiently large constant $C$.
 
 \item  $f_{1} (x)= f_{2} (x)$ for any $x \notin I$. 
 
 \item $f_{2}(x) +1 \ge f_{1} (x) \ge f_{2}(x)$ for any $x \in I$.
 
 \item $f_{1} (x) -f_{2} (x) \ge 1/2$ for any $x \in \frac{1}{2} I$. \end{itemize}

Define $Z_{1}, Z_{2} $ with respect to $f_{1} , f_{2}$. Applying \eqref{eqn:GZmean} to $Z_{j}$
with $T:= \alpha n^{1-c}$, for some small positive constant $\alpha$ (to be chosen),  we have 

\begin{equation} \label{eqn:Z} \P ( |Z_{j} - \E(Z_{j}) | \ge T ) 
\le 4 \exp(- \frac{cT^{2} |I|^{2}}{K^{2 } })=  \exp(-n^{\Omega (1)}), \end{equation}

\noindent given the fact that $|I|=O(n^{-c} )$ where $c$ is sufficiently small and that $K$ 
(the bound on the absolute values of the entries) is a sufficiently small power of $n$. 

By \eqref{eqn:largemean}
and choosing $\alpha$ sufficiently small, we can assume  $\E(Z_{1})- \E(Z_{2} )\ge 3\alpha n^{1-c}$.  By the triangle inequality and \eqref{eqn:Z}, it follows that 

$$\P( Z_{1} -Z_{2} \le  \alpha n^{1-c} )  \le   \exp(-n^{\Omega (1)}). $$

The desired bound \eqref{eqn:exponential} follows from this and the fact that 
 
 \begin{align*} Y_{c} &= \sum _{i=1}^{N} \chi_{I} (\lambda_{i}) \\
 &\ge \sum_{i=1} ^{N } f_{1} (\lambda_{i})-f_{2} (\lambda_{i} ) \\
 &= Z_{1} -Z_{2} . \end{align*}

{\it Acknowledgement.} We would like to thank P. Wood for the figures used in this paper and his careful reading and suggestions, O. Zeitouni for a  useful conversation 
concerning Appendix \ref{section:MP}, P. Forrester and B. Rider for enlightening remarks.

\end{document}